\documentclass{amsart}
\usepackage{amsmath}
\usepackage{amsthm}
\usepackage{listings}
\usepackage{xcolor}
\usepackage{mathtools}
\usepackage{paralist}
\usepackage{scalerel}
\usepackage{pbox}
\usepackage[noadjust]{cite}
\usepackage{thmtools,thm-restate}
\usepackage{latexsym}

\newtheorem{theorem}{Theorem}[section]
\newtheorem{proposition}[theorem]{Proposition}
\newtheorem{lemma}[theorem]{Lemma}
\usepackage{amssymb}

\usepackage{enumitem}  
\usepackage{stmaryrd}

\setenumerate{label={\normalfont(\textit{\roman*})}, ref={(\textrm{\roman*})},
wide
}

\usepackage{ifthen}
\usepackage{hyperref}
\usepackage[mathscr]{euscript}
\usepackage{lineno}

\newtheorem{corollary}[theorem]{Corollary}

\newtheorem{alphatheorem}{Theorem}

\theoremstyle{definition}
\newtheorem{definition}[theorem]{Definition}

\newtheorem{remark}[theorem]{Remark}

\DeclareMathOperator*{\EEE}{\scalerel*{\mathbb{E}}{\textstyle\sum}}

\newcommand{\inter}{\operatorname{int}}

\newcommand{\norm}[1]{\left\lVert #1 \right\rVert}

\newcommand{\normLip}[1]{\left\lVert #1 \right\rVert_{\mathrm{Lip}}}

\newcommand{\fp}[1]{\left\{ #1 \right\} }

\newcommand{\ip}[1]{\left\lfloor #1 \right\rfloor }

\newcommand{\floor}[1]{\left\lfloor #1 \right\rfloor}
\newcommand{\bra}[1]{\left( #1 \right)}
\newcommand{\brabig}[1]{\big( #1 \big)}

\renewcommand{\tilde}{\widetilde}
\renewcommand{\bar}{\overline}

\newcommand{\abs}[1]{\left|#1\right|}
\newcommand{\set}[2]{\left\{ #1 \ \middle| \ #2 \right\} }

\newcommand{\ceil}[1]{\left\lceil #1 \right\rceil}
\newcommand{\ffrac}[2]{#1/#2}

\renewcommand{\vec}[1]{#1}

\newcommand{\e}{\varepsilon}
\renewcommand{\a}{\alpha}
\renewcommand{\b}{\beta}

\newcommand{\NN}{\mathbb{N}}

\newcommand{\QQ}{\mathbb{Q}}

\newcommand{\EE}{\mathbb{E}}
\newcommand{\ZZ}{\mathbb{Z}}
\newcommand{\RR}{\mathbb{R}}
\newcommand{\CC}{\mathbb{C}}

\DeclareMathAlphabet{\mathpzc}{OT1}{pzc}{m}{it}

\newcommand{\cX}{\mathpzc{X}}

\newcommand{\cG}{\mathcal{G}}
\newcommand{\cH}{\mathcal{H}}

\newcommand{\cV}{\mathcal{V}}

\newcommand{\fg}{\mathfrak{g}}
\newcommand{\fh}{\mathfrak{h}}

\renewcommand{\k}{k}

\definecolor{checked}{RGB}{200,200,200}
\definecolor{final}{RGB}{ 255, 255, 255}
\definecolor{fresh}{RGB}{ 0, 128, 0}
\definecolor{wrong}{RGB}{ 255, 0, 0	}
\definecolor{Red}{RGB}{128,0,0}

\newcommand{\bb}{\mathbf}

\renewcommand{\subset}{\subseteq}

\newcommand*\patchAmsMathEnvironmentForLineno[1]{\expandafter\let\csname old#1\expandafter\endcsname\csname #1\endcsname
  \expandafter\let\csname oldend#1\expandafter\endcsname\csname end#1\endcsname
  \renewenvironment{#1}{\linenomath\csname old#1\endcsname}{\csname oldend#1\endcsname\endlinenomath}}\newcommand*\patchBothAmsMathEnvironmentsForLineno[1]{\patchAmsMathEnvironmentForLineno{#1}\patchAmsMathEnvironmentForLineno{#1*}}\AtBeginDocument{\patchBothAmsMathEnvironmentsForLineno{equation}\patchBothAmsMathEnvironmentsForLineno{align}\patchBothAmsMathEnvironmentsForLineno{flalign}\patchBothAmsMathEnvironmentsForLineno{alignat}\patchBothAmsMathEnvironmentsForLineno{gather}\patchBothAmsMathEnvironmentsForLineno{multline}}

\newcommand{\N}{\mathbb{N}}
\newcommand{\Z}{\mathbb{Z}}
\newcommand{\R}{\mathbb{R}}

\newcommand{\cb}[1]{\left\{#1 \right \}}

\newcommand{\rb}[1]{\left( #1 \right)}

\newcommand{\gp}{GP}

\newcommand{\pgp}{parametric GP}

\begin{document}

\author[J.\ Konieczny]{Jakub Konieczny}
\address{Universit\'e Claude Bernard Lyon 1, CNRS UMR 5208, Institut Camille Jordan,
F-69622 Villeurbanne Cedex, France}
\email{jakub.konieczny@gmail.com}

\author[C.\ M\"ullner]{Clemens M\"ullner}
\address{Institut f\"ur Diskrete Mathematik und Geometrie, TU Wien, Wiedner Hauptstr. 8--10, 1040 Wien, Austria}
\email{clemens.muellner@tuwien.ac.at}

\title[Bracket words along Hardy field sequences]{Bracket words along Hardy field sequences}
\date{\today}
 
\begin{abstract}
	We study bracket words, which are a far-reaching generalisation of Sturmian words, along Hardy field sequences, which are a far-reaching generalisation of Piatetski--Shapiro sequences $\lfloor n^c \rfloor$. We show that thus obtained sequences are deterministic (i.e., they have sub-exponential subword complexity) and satisfy Sarnak's conjecture.
\end{abstract}

\subjclass[2010]{}

\maketitle 
\tableofcontents

\section{Introduction}
\label{sec:intro}

One of the key results in a recent paper~\cite{Deshouillers2022} by J.-M.\ Deshouillers, M.\ Drmota, A.\ Shubin, L.\ Spiegelhofer and the second-named author states that the subword complexity of $(\floor{n^c} \bmod m)_{n=0}^{\infty}$ grows at most polynomially, which in particular shows that this sequence is deterministic.
The philosophy behind this result is the following: if we take a regularly growing function ($(\floor{n^c})_{n=0}^{\infty}$) and apply a very simple rule to it (taking the residue modulo $m$), then the resulting sequence is still quite simple (in this case it has polynomial subword complexity).
In this paper we vastly generalize both main aspects of this result, i.e. we replace $(\floor{n^c})_{n=0}^{\infty}$ with Hardy sequences and we replace taking the residue modulo $m$ by applying a bracket word.

Sturmian words are among the simplest and most extensively studied classes of infinite words over a finite alphabet. One of their defining properties is extremely low subword complexity. Recall that the subword complexity of an infinite word $\bb a = (a(n))_{n=0}^\infty$ over a finite alphabet $\Sigma$ is the function $p_a$ which assigns to each integer $N$ the number $p_{\bb a}(N)$ of words $w \in \Sigma^N$ which appear in ${\bb a}$. If there exists at least one value of $N$ such that $p_{\bb a}(N) \leq N$ then $a$ must be eventually periodic, in which case $p_a$ is bounded. If ${\bb a}$ is a Sturmian word then $p_{\bb a}(N) = N+1$ for all $N$, which in light of the remark above is the least subword complexity possible for a word that is not eventually periodic.

In \cite{AdamczewskiKonieczny-2022} B.\ Adamczewski and the first-named author studied a generalisation of Sturmian words obtained by considering letter-to-letter codings of finitely-valued generalised polynomials, which they dubbed \emph{bracket words}. A \emph{generalised polynomial} is an expression built from the usual polynomials using addition, multiplication and the integer part function. For instance, Sturmian words (up to letter-to-letter coding) take the form 
\[
	a(n) = \ip{\a (n+1) + \b} - \ip{\a n + \b}
\]
with $\a \in (0,1) \setminus \QQ$ and $b \in (0,1)$ (possibly with the integer part $\ip{\cdot}$ replaced by the ceiling $\lceil{\cdot}\rceil$) 
, and hence are special cases of bracket words. One of the main results of \cite{AdamczewskiKonieczny-2022} is a polynomial bound on subword complexity of bracket words: $p_{a}(N) \ll N^C$ for a constant $C$ (dependent on $a$).

In \cite{Deshouillers2022}, J.-M.\ Deshouillers, M.\ Drmota, A.\ Shubin, L.\ Spiegelhofer and the second-named author investigated synchronising automatic sequences along Piatetski--Shapiro sequences $(\ip{n^c})_{n=0}^\infty$, where $c > 1$. A special case which plays a crucial role in the argument is when the synchronising automatic sequence is periodic, in which case they obtained a polynomial bound on the subword complexity. 

As a joint extension of the two lines of investigation discussed above, we investigate bracket words along Piatetski--Shapiro sequences. In fact, we can deal with a considerably larger class of Hardy field functions with polynomial growth, which in addition to $n^c$ ($c > 1$) include logarithmic-exponential expressions such as $\a n^{c} + \a' n^{c'}$ or $n^c \log^{c'} n$, as well as some more complicated expressions such as $\log(n!)$.
Our first result is a bound on the subword complexity.
\begin{alphatheorem}\label{thm:A}
	Let ${\bb a} = (a(n))_{n \in \ZZ}$ be a (two-sided) bracket word over the alphabet $\Sigma$ and let $f \colon \RR_+ \to \RR$ be a Hardy field function with polynomial growth. Then the subword complexity of $(a(\ip{f(n)})_{n = 0}^\infty$ is bounded by $\exp(O(H^{\delta}))$ for some $0<\delta <1$.\end{alphatheorem}

The study of (special) automatic sequences along Piatetski-Shapiro sequences $\floor{ n^c }$ has a long history. 
We mention results by C. \ Mauduit and J. \ Rivat~\cite{Mauduit1995, Mauduit2005}, by J.-M. \ Deshouillers, M. \ Drmota, and J. \ Morgenbesser~\cite{Deshouillers2012}, by L. \ Spiegelhofer~\cite{Spiegelhofer2015,Spiegelhofer2020} and by L. \ Spiegelhofer and the second-named author~\cite{Muellner2017a}.
Interestingly there can appear two very different situations: 
On the one hand, the Thue-Morse sequence along Piatetski-Shapiro sequences (for $1<c<3/2$) is normal --- in particular it has maximal subword complexity.
On the other hand, synchronizing automatic sequences along Piatetski-Shapiro sequences are very far from normal - they have subexponential subword complexity.
One natural generalization of automatic sequences are morphic sequences. These are letter-to-letter codings of fixed points of substitutions.
A very prominent morphic sequence is the Fibonacci word which is the fixed point of the substitution $0 \mapsto 01, 1 \mapsto 0$.
Moreover, this sequence is also a Sturmian word and many interesting morphic sequences are also Sturmian words (see for example~\cite{Klouda2018}).
Thus, we obtain as a very special case (one of) the first results for morphic sequences along Piatetski-Shapiro sequences.

It follows from Theorem \ref{thm:A} that the sequence $(a(\ip{f(n)})_{n = 0}^\infty$ is deterministic, meaning that it has subexponential subword-complexity. A conjecture of Sarnak \cite{Sarnak2011} asserts that each deterministic sequence should be orthogonal to the M\"obius function, given by
\[
	\mu(n) = 
	\begin{cases}
		(-1)^k &\text{if $n$ is the product of $k$ distinct primes;}\\
		0 &\text{if $n$ is divisible by a square.}
	\end{cases}
\]
This conjecture in general is wide open. However, it has been resolved
in a number of special cases \cite{Bourgain2013, Bourgain2013a, Deshouillers2015, Downarowicz2015, ElAbdalaoui2016, ElAbdalaoui2014, Ferenczi2016, GreenTao-2012-Mobius, Green2012, KulagaPrzymus2015, Liu2015, Mauduit2010, Mauduit2015, Muellner2017, Peckner2018, Veech2016}, see also the recent survey articles \cite{Drmota,Ferenczi2018}. Of particular importance to the current paper is  M\"obius orthogonality for nilsequences \cite{GreenTao-2012-Mobius}, which was recently strengthened to short intervals \cite{MSTT}. As we discuss later in the paper, this is closely connected to bracket words thanks to the work of Bergelson and Leibman \cite{BergelsonLeibman-2007}.
Our second result is the M\"obius orthogonality for bracket words along Hardy field functions.
\begin{alphatheorem}\label{thm:B}
	Let ${\bb a} = (a(n))_{n \in \ZZ}$ be a (two-sided) $\RR$-valued bracket word and let $f \colon \RR_+ \to \RR$ be a Hardy field function with polynomial growth. Then 
\begin{equation}\label{eq:intro:thm-B}
	\frac{1}{N} \sum_{n=1}^N \mu(n) a(n) \to 0 \text{ as } N \to \infty.
\end{equation}
\end{alphatheorem}

\begin{remark}
We point out that using similar techniques, it is possible to obtain a slightly stronger result. Firstly, instead of the bracket word, we could work with a bounded generalised polynomial; in fact, each bounded generalised polynomial can be approximated in the supremum norm by finitely-valued ones, which allows for a straightforward reduction. Secondly, since all of the key ingredients in the proof of Theorem \ref{thm:B} are quantitative, one can obtain explicit rate of convergence to $0$ in \eqref{eq:intro:thm-B}. We leave the details to the interested reader.
\end{remark}

Theorem \ref{thm:B} is closely related to M\"{o}bius orthogonality for nilsequences, that is, sequences that can be obtained by evaluating a continuous function along an orbit of a point in a nilsystem. The connection between generalised polynomials and nilsequences was established by Bergelson and Leibman \cite{BergelsonLeibman-2007}, who showed that bounded generalised polynomials can be represented by evaluating a piecewise polynomial function along an orbit in a nilsystem (see Theorem \ref{thm:BL} for details). 

The fact that nilsequences are orthogonal to the M\"obius function was established by Green and Tao \cite{GreenTao-2012-Mobius} as a part of their program of understanding additive patterns in the primes. In fact, \cite{GreenTao-2012-Mobius} already contains an outline of the proof of M\"{o}bius orthogonality for bounded generalised polynomials, although some technical details are left out.

In order to obtain a result for a bracket word along a Hardy field function, we split the range of summation into intervals where the Hardy field function under consideration can be efficiently approximated by polynomials. We are then left with the task of establishing cancellation in each of these intervals. A key ingredient is M\"obius orthogonality for nilsequences in short intervals, recently established in \cite{MSTT}, Theorem \ref{thm:Mobius:MSTT}. The main technical difficulty of our argument lies in extending Theorem \ref{thm:Mobius:MSTT} to piecewise constant (and hence necessarily not continuous) functions with semialgebraic pieces, which we accomplish in Section \ref{sec:short-int}.

\subsection{Plan of the paper} 
In Section~\ref{sec:hardy} we recall some basic definitions and results about Hardy fields. Moreover, we study Taylor polynomials of functions from a Hardy field which generalizes the corresponding part in~\cite{Deshouillers2022}. This allows us to locally replace functions from a Hardy field with polynomials. Thus, we need to be able to work with polynomials with varying coefficients.
To do so, we study in Section~\ref{sec:genpoly} parametric generalised polynomials which builds on and refines results obtained in~\cite{AdamczewskiKonieczny-2022}. These tools allow us to prove Theorem~\ref{thm:A}.
In Section~\ref{sec:nilmanifolds} we present some basics on nilmanifolds and discuss the connection to generalized polynomials.
Then, in Section~\ref{sec:mobius} we recall a result on M\"{o}bius orthogonality for nilsequences in short intervals. This is the final result that we need to prove Theorem~\ref{thm:B}. One naturally arising difficulty is to translate the result on M\"obius orthogonality for smooth functions to piecewise polynomial functions instead.

\subsection*{Notation}
We use $\NN = \{1,2,\dots\}$ to denote the set of positive integers and $\NN_0 = \NN \cup \{0\}$. For $N \in \NN$, we let $[N] = \{0,1,\dots,N-1\}$. 
For a non-empty finite set $X$ and a map $f \colon X \to \RR$, we use the symbol $\EE$ borrowed from probability theory to denote the average
$ \EE_{x \in X} f(x) = \frac{1}{\abs{X}} \sum_{x \in X} f(x)$.

\subsection*{Acknowledgements}
The authors wish to thank Michael Drmota for many insightful discussions, for suggesting this problem, and also for inviting the first-named author to Vienna for a visit during which this project started; and Fernando Xuancheng Shao for helpful comments on M\"obius orthogonality of nilsequences.

The first-named author works within the framework of the LABEX MILYON (ANR-10-LABX-0070) of Universit\'e de Lyon, within the program "Investissements d'Avenir" (ANR-11-IDEX-0007) operated by the French National Research Agency (ANR).
 The second-named author is supported by the Austrian-French project “Arithmetic Randomness” between FWF and ANR (grant numbers I4945-N and ANR-20-CE91-0006).

\section{Hardy fields}\label{sec:hardy}
In this section we discuss functions from a Hardy field which have polynomial growth. In particular we study how the Taylor-polynomial of $f$ can be used to describe $\ip{f(n)}$.
Therefore, we first gather some basic results on Hardy fields. Then we discus the uniform distribution of polynomials modulo $\Z$.
Finally, we study properties of Taylor polynomials and prove the main theorem of this section, namely Theorem~\ref{thm:error}.

\subsection{Preliminaries}
We start by gathering the basic facts and results on Hardy fields. For further discussion we refer e.g.\ to \cite{Boshernitzan-1994} and \cite{Frantzikinakis-2009}.

Let $\mathcal{B}$ be the collection of equivalence classes of real valued functions
defined on some half line $(c, \infty)$, where we identify two functions if they agree eventually.\footnote{The equivalence classes just defined are often called \emph{germs of functions}. We choose to refer to elements of $\mathcal{B}$ as functions instead, with the understanding that all the operations defined and statements made for elements of $\mathcal{B}$ are considered only for sufficiently large values of $t \in \R$.} A \emph{Hardy field} $H$ is a subfield of the ring $(\mathcal{B}, + , \cdot)$ that is closed under differentiation, meaning that $H$ is a subring of $\mathcal{B}$ such that for each $0 \neq f \in H$, the inverse $1/f$ exists and belongs to $H$, $f$ is differentiable and $f' \in H$. We let $\cH$ denote the union of all Hardy fields. 
If $f \in \cH$ is defined on $[0,\infty)$ (one can always choose such
a representative of $f$) we call the sequence $(f(n))_{n=0}^\infty$ a \emph{Hardy sequence}.

We note that choosing different representatives of the same germ of a function $f$, changes the number of subwords of length $N$ of $a(\ip{f(n)})$ by at most an additive constant. As a consequence, the asymptotic behaviour of the subword complexity of $a(\ip{f(n)})$ depends only on the germ of $f$.

A \emph{logarithmic-exponential function} is any real-valued function on a half-line $(c,\infty)$ that can be constructed from the identity map $t \mapsto t$ using basic arithmetic operations $+,-,\times, :$, the logarithmic and the exponential functions, and real constants. For example, $ t^2 + 5t, t^{\sqrt{2}+\sqrt{3}}, e^{(\log t)^2}$ and $e^{\sqrt{\log t}}/\sqrt{t^2+1}$ are all logarithmic-exponential functions.
Every logarithmic-exponential functions belongs to $\cH$, and so do some other classical functions such as $\Gamma$, $\zeta$ or $t \mapsto \sin(1/t)$.

For real-valued functions $f$ and $g$ on $(c,\infty)$ such that $g(t)$ is non-zero for sufficiently large $t$, we write $f(t) \prec g(t)$ if $\lim_{t\to \infty} f(t) / g(t) = 0$, $f(t) \sim g(t)$ if $\lim_{t\to \infty} f(t) / g(t)$ is a non-zero real number and $f(t) \ll g(t)$ if there exists $C>0$ such that $\abs{f(t)} \leq C\abs{g(t)}$ for all large $t$. For completeness, we let $0 \sim 0$ and $0 \ll 0$. 

We state the following well-known facts as lemmas.
\begin{lemma}\label{lem:Hardy_event}
	Let $f \in \cH$ be a function that is not eventually zero. Then $f$ is eventually strictly positive or negative.
	If $f$ is not eventually constant, then $f$ is eventually strictly monotone.
\end{lemma}
\begin{proof}
	Since $f$ is not eventually $0$, there exists the inverse function $1/f$ --- in particular, $f(t) \neq 0$ for $t$ large enough. Now, the first part follows from continuity of $f$. The second part follows directly from the first part by considering $f'$.
\end{proof}

\begin{lemma}\label{lem:Hardy:trichotomy}
	Let $H$ be a Hardy field and let $f,g \in H$. Then one of the following holds: $f \prec g$, $f \sim g$ or $f \succ g$.
\end{lemma}
\begin{proof}
	If $g$ is eventually zero, the situation is trivial, so assume that this is not the case. Since $f/g$ is eventually monotone, the limit $\lim_{t \to \infty} \abs{f(t)}/\abs{g(t)} \in \RR \cup \{\infty\}$ exists. If the limit is infinite then $f \succ g$. If the limit is zero then $f \prec g$. If the limit is finite and non-zero then $f \sim g$.
\end{proof}

\begin{definition}
	We say that $f$ has \emph{polynomial growth} if there exists $n \in \N$ such that $f(t) \prec t^n$.
\end{definition}

We will make use of the following estimates for the derivatives of functions with polynomial growth.
\begin{lemma}[{\cite[Lem.\ 2.1]{Frantzikinakis-2009}}]\label{lem:f'(t)-cases}
	Let $f \in \mathcal{H}$ be a function with polynomial growth. Then at least one of the following holds:
	\begin{enumerate}
\item $f(t) \prec t^{-n}$ for all $n \in \NN$;
	\item $f(t) \to c \neq 0$ as $t \to \infty$ for some constant $c$;
	\item $f(t)/(t (\log t)^2)\prec f'(t) \ll f(t)/t$.
	\end{enumerate}
\end{lemma}

\begin{lemma}\label{lem:f(t)<<t^(-omega)}
	Let $f \in \cH$ be a function such that $f(t) \prec t^{-n}$ for all $n \in \NN$. Then also $f^{(\ell)}(n) \prec t^{-n}$ for all $\ell,n \in \NN$.
\end{lemma}
\begin{proof}
	Reasoning inductively, it is enough to consider the case where $\ell = 1$. Suppose, for the sake of contradiction, that $\abs{f'(t)} \gg t^{-n}$ for some $n \in \NN$. Since $f(t) \to 0$ as $t \to \infty$ and since $f$ is eventually monotone, for sufficiently large $t$ we have
	\[
		\abs{f(t)} = \int_{t}^{\infty} \abs{f'(s)} ds \gg \int_{t}^{\infty} s^{-n} ds \gg t^{-n+1},
	\]
	contradicting the assumption on $f$.
\end{proof}

\begin{lemma}\label{lem:f'(t)<<t^(k-1)}
	Let $f \in \cH$ and assume that $f(t) \ll t^{\k}$ for some $\k \in \Z$. Then $f^{(\ell)}(t) \ll t^{\k-\ell}$ for each $\ell \in \NN$.
\end{lemma}
\begin{proof}
	Reasoning inductively, it is enough to consider the case where $\ell = 1$. We consider the three possibilities in Lemma \ref{lem:f'(t)-cases}. If $f(t) \prec t^{-n}$ for all $n \in \NN$ then the claim is trivially true by Lemma \ref{lem:f(t)<<t^(-omega)}. If $f'(t) \ll f(t)/t$ then $f'(t) \ll t^{\k-1}$, as needed. Finally, suppose that $f(t) \to c \neq 0$ as $n \to \infty$. Clearly, in this case $\k \geq 0$. We may decompose $f(t) = \bar f(t) + c$, where $\bar f(t) = f(t) - c$ and $\bar f(t) \prec 1$.
Repeating the reasoning with $\bar f$ in place of $f$ we conclude that $f'(t) = \bar f'(t) \ll t^{-1} \ll t^{\k - 1}$.
\end{proof}
\begin{remark}
For each $f \in \cH$ and each logarithmic-exponential function $g$, there exists a Hardy field $H$ such that $f,g \in H$ (see e.g.\ \cite{Boshernitzan-1994}). Hence, it follows from Lemma \ref{lem:Hardy:trichotomy} that for each $f \in \cH$ there exists $\k_0(f) \in \Z \cup \{- \infty,+\infty\}$ such that, for $\k \in \Z$ we have: $f(t) \prec t^\k$ if $\k > \k_0(f)$, $f(t) \succ t^\k$ if $\k < \k_0(f)$ and, if $k_0(f)$ is finite, $f(t)\ll t^{\k_0(f)}$. Lemma \ref{lem:f'(t)<<t^(k-1)} implies that $\k_0(f^{(\ell)}) \leq \k_0(f) - \ell$ (with the convention that $\pm \infty - \ell = \pm \infty$).
\end{remark}

\subsection{Uniform distribution of polynomials}
In this subsection we recall a result about the uniform distribution of polynomials modulo $\Z$ which we need for the next subsection about Taylor-polynomials.
It is well-known that a polynomial distributes uniformly modulo $\Z$ if and only if at least one (non-constant) coefficient is irrational.
The following proposition is a quantitative version of this statement.

First we need to specify the way we quantify how uniformly distributed a sequence $a(n) \bmod \Z$ is:
Let $( x_1, \ldots, x_N )$ be a finite sequence of real numbers. Its \emph{discrepancy} is defined by
\begin{align}\label{eq_discrepancy}
	D_N (x_1, \ldots, x_N) = \sup_{0 \le \alpha \le \beta \le 1} \biggl| \frac{\# \{ n \le N : \alpha \le \{ x_n \} < \beta \}}{N} - (\beta - \alpha) \biggr|. 
\end{align}

Thus, we have the necessary prerequisites to state the following proposition. 
\begin{proposition}[Proposition 5.2 in~\cite{Deshouillers2022}]\label{pr_approx}
	Suppose that $g: \Z \to \R$ is a polynomial of degree $d$, which we write as
	\begin{align*}
		g(n) = \beta_0 + n \beta_1 + \ldots + n^d \beta_d.
	\end{align*}
	Furthermore, let $\delta \in (0,1/2)$. Then either the discrepancy of $(g(n) \bmod \Z)_{n\in [N]}$ is smaller than $\delta$, or else there is an integer $1\leq \ell \ll \delta^{-O_d(1)}$, such that
	\begin{align*}
		\max_{1\leq j \leq d} N^j \norm{\ell \beta_j} \ll \delta^{-O_d(1)}.
	\end{align*}
\end{proposition}
This proposition is a direct consequence of Proposition 4.3 in \cite{GreenTao-2012}, who attribute this result to Weyl.

\subsection{Taylor expansions}
For any germ $f \in \mathcal{H}$ we consider a representative that is defined on $[1,\infty)$ and also call it $f$. Then, for any $x \in (1,\infty)$ and $\ell \in \NN_0$ we can consider the length-$\ell$ Taylor expansion of $f$ at the point $x$,
\begin{align}
\label{eq:Taylor-1}
	f(x + y) &= P_{x,\ell}(y) + R_{x,\ell}(y),\\
\label{eq:Taylor-2}
	P_{x,\ell}(y) &:=	 f(x) + y f'(x) + \ldots + \frac{y^{\ell-1}}{(\ell-1)!} f^{(\ell-1)}(x),\\
\label{eq:Taylor-3}
	R_{x,\ell}(y) &:= \frac{y^{\ell}}{\ell!} f^{(\ell)}\bra{x + \xi_{\ell}(N,h)}, \text{ where } \xi_{\ell}(x,y) \in [0, y].
\end{align}

\begin{proposition}\label{prop:Taylor}
	Let $\k \in \Z$, $\ell \in \NN_0$, and let $f \in \cH$ be a function with $f(t) \ll t^{\k}$. Then the error term $R_{x,\ell}(y)$ in the Taylor expansion \eqref{eq:Taylor-1}--\eqref{eq:Taylor-3} satisfies
	\[
		R_{x,\ell}(y) \ll y^{\ell} x^{\k-\ell}
	\]
	uniformly for all $x\geq 1$ and $0\leq y \leq x$, where the implied constant only depends on $f$ and $\ell$.
\end{proposition}

\begin{proof}
	Combining \eqref{eq:Taylor-3} and Lemma \ref{lem:f'(t)<<t^(k-1)} we have 
	\begin{align*}
	y^{-\ell} R_{x,\ell}(y) \ll \sup_{\xi \in [0,y]} f^{(\ell)}(x+\xi) \ll \sup_{\xi \in [0,y]} (x+\xi)^{\k-\ell} =
	\begin{cases}
		x^{\k-\ell} & \text{if } \k < \ell;\\
		(x+y)^{\k-\ell} & \text{if } \k \geq \ell.
	\end{cases} 
	\end{align*}
	Assuming that $x \geq y$, the two estimates are equivalent.
\end{proof}

\begin{lemma}\label{lem:mono}
	Let $k \in \NN$ and let $f$ be a $k$ times continuously differentiable function defined on an open interval $I \subset \RR$. Suppose that $f^{(k)}(t)$ has constant sign on $I$.
	Then $f$ changes monotonicity on $I$ at most $k-1$ times.
\end{lemma}
\begin{proof}
	If $f^{(k)}(t)$ is constant zero for all $t \in I$, then $f$ is a polynomial of degree at most $k-1$ and the statement is trivially true.
	Thus, we assume without loss of generality that $f^{(k)}(t) > 0$ for all $t \in I$.
	Let us assume for the sake of contradiction that $f$ changes monotonicity at least $k$ times. 
	Thus, $f'$ has at least $k$ zeros in $I$.
	It follows from the mean value theorem that $f''$ has at least $k-1$ zeros in $I$.
	Inductively applying this reasoning shows that $f^{(k)}$ has at least $1$ zero in $I$ giving the desired contradiction.
\end{proof}

\begin{theorem}\label{thm:error}
	Let $k,\ell \in \NN$ be integers with $k < \ell$ and let $f \in \mathcal{H}$ be a function satisfying $f(t) \ll t^{k}$, and let $P_{N,\ell}$ and $R_{N,\ell}$ be given by \eqref{eq:Taylor-1}--\eqref{eq:Taylor-3}. 
Then there exists some $0<\eta<1$ (only depending on $\ell$) such that for any $H \in \NN$, the formula	
	\begin{align}\label{eq_goal}
		e_{N}(h) &:= \floor{f(N+h)} - \floor{P_{N,\ell}(h)}, &&& 0 \leq h < H.
	\end{align}
	defines at most $\exp(O(H^{\eta}))$ different functions $e_{N}: [H] \to \ZZ$ for $N \in \NN$.
Moreover, for each $N$, at least one of the following holds
\begin{enumerate}
\item\label{it:error:A} \textit{$N$ is small:} $N = O(H^{\bra{\ell + \eta}/\bra{\ell - k}})$. 
\item\label{it:error:B} \textit{$e_N$ is sparse:} There are at most $O(H^{\eta})$ values of $h \in [H]$ such that $e_N(h) \neq 0$.
\item\label{it:error:C} \textit{$e_N$ is structured:} There exists a partition of $[H]$ into $O(H^{\eta})$ arithmetic progressions with step $O(H^{\eta})$ on which $e_N$ is constant.
\end{enumerate}
\end{theorem}
(In the theorem above, the constants implicit in the $O(\cdot)$ notation are allowed to depend on $k,\ell$ and $f$.)

\begin{proof}
	We define $\e = H^{\eta_0}$ for some $\eta_0 > 0$ which only depends on $\ell$ and will be specified later.
	Let $N \in \NN$. Recall that by Proposition \ref{prop:Taylor}, we have 
\begin{align}\label{eq:Hardy:85:1}
	\abs{R_{N,\ell}(h)} &\leq \e &&&\text{for all }  0 &\leq h < H
\end{align}	
	 unless $N \ll \e^{-{1}/\bra{\ell-k} }H^{{\ell}/\bra{\ell - k}} = H^{\bra{\ell + \eta_0}/\bra{\ell-k}}$. Thus, the values of $N$ such that \eqref{eq:Hardy:85:1} is false contribute only $O\bra{H^{O(1)}}$ different sequences $e_N$, and we may freely assume that $N$ is large enough that \eqref{eq:Hardy:85:1} holds. In this case we have $e_N: [H] \to \{-1,0,1\}$. Additionally, by Lemma \ref{lem:Hardy_event} we may also assume that $f^{(\ell)}(x) \neq 0$ for all $x \geq N$. As a consequence of \eqref{eq:Hardy:85:1}, for each $0 \leq h < H$, if 
	\begin{align}\label{eq_far_away}
		\varepsilon < \cb{P_{N,\ell}(h)} < 1- \varepsilon
	\end{align}
	then \(	\floor{f(N+h)} = \floor{P_{N,\ell}(h)} \) and hence $e_N(h) = 0$.
	
	Let $\a_0,\dots,\a_{\ell-1}$ denote the coefficients of $P_{N,\ell}$:
	\[
		P_{N,\ell}(h) = \a_0 + \a_1 h + \dots + \a_{\ell-1}h^{\ell-1}.
	\]	
	By Proposition~\ref{pr_approx}, we distinguish two cases.
	\begin{enumerate}
		\item $(P_{N,\ell}(h))_{h \in [H]}$ has discrepancy at most $\varepsilon$.
		\item There exists $1\leq q \ll \varepsilon^{-O(1)}$ such that $\max_{0 \leq j < \ell} H^j\norm{q \a_j} \ll \e^{-O(1)}$.
	\end{enumerate}
	
	In the first case, it follows that the number of $h \in [H]$ such that  \eqref{eq_far_away} does not hold is at most ${3}\varepsilon H$.
	Thus, $e_{N}$ is sparse, i.e. it has at most ${3}\varepsilon H \ll H^{1-\eta_0}$ non-zero entries.
	It remains to estimate the number of the sequences $e_N$ of this type. Using a standard estimate 
$\binom{n}{k} \leq \ffrac{n^k}{k!} < \ffrac{(en)^k}{k^k}$
	we find
	\begin{align*}
		\log \bra{ \sum_{0\leq j \leq 3 \varepsilon H} \binom{H}{j} 2^j } 
		& \ll \log \bra{ 3\varepsilon H} + \log \binom{H}{ 3 \varepsilon H} +  3 \varepsilon H \\
& \ll \log(3 H^{1-\eta_0}) + 3 \e H \log(e 3 H^{1-\eta_0}) + 3 H^{1-\eta_0}\\
		&\ll_{\eta_0} H^{1-\eta_0/2}.
	\end{align*}
Thus the number of distinct sequences $e_N$ is bounded by $\exp(O(H^{1-\eta_0/2}))$, which gives the desired result as long as $1-\eta_0/2\leq\eta$.
	
	In the second case we split $[H]$ into arithmetic progressions with common difference $q \ll \varepsilon^{-O_{\ell}(1)}$.
	This allows us to write (for $0\leq m < q$)
	\begin{align*}
		P_{N,\ell}(q h + m) &= \alpha_{0} + (q h + m) \alpha_{1} + \ldots + (q h + m)^{\ell-1} \alpha_{\ell-1}\\
			&= \beta_{0} + h \beta_{1} + \ldots + h^{\ell-1}\beta_{\ell-1}.
	\end{align*}
	The defining property of $q$ implies that
	\begin{align*}
		\max_{1\leq j < \ell} H^j \norm{\beta_{j}} \ll \varepsilon^{-O_{\ell}(1)}.
	\end{align*}
	In particular, we can write
	\begin{align*}
		\beta_{j} = z_{j} + s_{j},
	\end{align*}
	where $z_{j} \in \Z$ and $\abs{s_{j}} \ll H^{-j} \cdot \e^{-O_{\ell}(1)}$ for $ 0 \leq j < \ell$.
	Putting everything together, we find
	\begin{align*}
		f(N+q h +m) = Q(h) + r(h) + R_{N,\ell}(q h + m),
	\end{align*}
	where
	\begin{align*}
		Q(h) &= z_{0} + h z_{1} + \ldots + h^{\ell-1} z_{\ell-1}\\
		r(h) &= s_{0} + h s_{1} + \ldots + h^{\ell-1} s_{\ell-1}.
	\end{align*}
	In particular, $Q$ is a polynomial of degree at most $\ell-1$ with integer coefficients and $P_{N,
\ell}(q h +m) = Q(h) + r(h)$. Moreover,  $\abs{r(h)} \ll \e^{-O_{\ell}(1)}$ for all $h \in [0,H/q]$.
	Since $\abs{R_{N,\ell}(h)} \leq \e$, we see that 
	\begin{align*}
		\floor{f(N+q h + m)} \neq \floor{P_{N,\ell}(q h+m)}
	\end{align*}
	holds exactly if either
	\begin{align}
	\begin{split}\label{eq_conditions}
		\cb{r(h)} \leq \varepsilon \quad &\text{and} \quad \cb{r(h) + R_{N,\ell}(q h + m)} \geq 1-\varepsilon, \quad \text{or}\\
		\cb{r(h)} \geq 1-\varepsilon \quad &\text{and} \quad \cb{r(h) + R_{N,\ell}(q h + m)} \leq \varepsilon.
	\end{split}
	\end{align}
	In the first case $e_N(q h +m) = 1$ and in the second case $e_N(q h + m) = -1$.
	Since $r(h)$ is a polynomial of degree at most $\ell -1$, it changes monotonicity at most $\ell - 2$ times. Since the $\ell$-th derivative of $r(h) + R_{N,\ell}(q h + m) = f(N + q h + m) - P_{N,\ell}(q h + m) + r(h)$ has constant sign, by Lemma \ref{lem:mono} it changes monotonicity at most $\ell-1$ times on the interval $[0, H/q]$. Hence, we can decompose $[0, H/q]$ into at most $2 \ell-2$ intervals  $I_1, \ldots, I_p$ on which $r(h)$ and $r(h) + R_{N,\ell}(q h + m)$ are both monotone. 
	As $\abs{r(h)} \ll \e^{-O_{\ell}(1)}$, we can further subdivide each of the intervals $I_j$ into $O(\e^{-O_{\ell}(1)})$ subintervals such that for each subinterval, each of the inequalities in either true on the entire subinterval or false on the entire subinterval. As a consequence, $e_N$ is structured, i.e., $e_N$ is constant on each subinterval.
Thus, we have found a decomposition of $[H]$ into $O(\e^{-O_{\ell}(1)})$ arithmetic progressions on which $e_N$ is constant.
We can write $O(\e^{-O_{\ell}(1)}) = O(H^{C \eta_0})$ for some $C = C(\ell)>0$.
 Using the rough estimate $H^3$ for the number of arithmetic sequences contained in $[H]$, we can bound the number of sequences $e_N$ which arise this way by
\[
	(H^3)^{O(H^{C \eta_0})} = \exp\bra{O( H^{C \eta_0} \log H) } =  \exp\bra{ O_{\eta_0}(H^{(C+1) \eta_0 }}. 
\]
It  remains to choose $\eta_0 = (C+2)^{-1}$ and $\eta = 1-(2(C+2))^{-1}$ to finish the proof.
\end{proof} 
\section{Parametric generalised polynomials}
\label{sec:genpoly}
 
In this section we discuss parametric generalised polynomials which builds on and refines results obtained in~\cite{AdamczewskiKonieczny-2022}. 
In particular, we show that for any parametrised general polynomial that takes values in $[M]$, we can assume that the parameters belong to $[0,1)^J$ for some finite set $J$ (Proposition~\ref{prop:param-reduce-bdd}).
This allows us to show a polynomial bound on the number of subwords of bracket words along polynomials of a fixed degree (Corollary~\ref{cor:bracket_word_along_polys}).
At the end of the section we give the proof of Theorem~\ref{thm:A}.

Let $d \in \NN$. Generalised polynomial maps (or {\gp} maps for short) from $\RR^d$ to $\RR$ are the smallest family $\cG$ such that 
\begin{inparaenum}
\item all polynomial maps belong to $\cG$;
\item if $g,h \in \cG$ then also $ g + h, g \cdot h \in \cG$ (with operations defined pointwise);
\item if $g \in \cG$ then also $\ip{g} \in \cG$, where $\ip{g}$ is defined pointwise: $\ip{g}(x) = \ip{g(x)}$.
\end{inparaenum}
We note that generalised polynomials maps are also closed under the operation of taking the fractional part, given by $\fp{g} = g - \ip{g}$.
For a sets $\Omega \subset \RR^d$ and $\Sigma \subset \RR$ (e.g., $\Omega = \ZZ^d$, $\Sigma = \ZZ$), by a generalised polynomial map $g \colon \Omega \to \Sigma$ we mean the restriction $\tilde g|_{\Omega}$ to $\Omega$ of a generalised polynomial map $\tilde g \colon \RR^d \to \RR$ such that $\tilde g(\Omega) \subset \Sigma$. We point out that, unlike in the case of polynomials, the lift $\tilde g$ is not uniquely determined by $g$, unless $\Omega = \RR^d$.
 
In \cite{AdamczewskiKonieczny-2022}, we introduced a notion of a \emph{parametric {\gp} map} $\ZZ \to \RR$ with a finite index set $I$, which (modulo some notational conventions) is essentially the same as a {\gp} map $\RR^I \times \ZZ \to \RR$. For instance, the formula
\begin{align*}
	g_{\a,\b}(n) &= \ip{ \a n \ip{\b n}+\sqrt{2}n^2} &&& (\a,\b \in \RR)
\end{align*}
defines a {{\gp} map} $\ZZ \to \RR$ (or, strictly speaking, a family of {\gp} maps) parametrised by $\RR^2$. Formally, a \emph{parametric {\gp} map with index set $I$} or a \emph{{\gp} map parametrised by $\RR^I$} is a map $\RR^I \to \RR^{\ZZ}$, $\a \mapsto g_\a$, such that the combined map $\RR^I \times \ZZ \to \RR$, $(\a,n) \mapsto g_\a(n)$, is a {\gp} map.

Here, we will need a marginally more precise notion, where the set of parameters takes the form $\RR^{I_{\mathrm{real}}} \times \ZZ^{I_{\mathrm{int}}} \times [0,1)^{I_{\mathrm{frac}}}$ rather than $\RR^I$. Let $I_{\mathrm{real}}, I_{\mathrm{int}}, I_{\mathrm{frac}}$ be pairwise disjoint finite sets and put $I = I_{\mathrm{real}} \cup I_{\mathrm{int}} \cup I_{\mathrm{frac}}$. Then a \emph{{\gp} map parametrised by $\RR^{I_{\mathrm{real}}} \times \ZZ^{I_{\mathrm{int}}} \times [0,1)^{I_{\mathrm{frac}}}$} is the restriction of a {{\gp} map parametrised by $\RR^{I_{\mathrm{real}}} \times \RR^{I_{\mathrm{int}}} \times \RR^{I_{\mathrm{frac}}}$} (as defined above) to $\RR^{I_{\mathrm{real}}} \times \ZZ^{I_{\mathrm{int}}} \times [0,1)^{I_{\mathrm{frac}}}$.  We note that in the case where $I_{\mathrm{int}} = I_{\mathrm{frac}} = \emptyset$, the new definition is consistent with the previous one.

In \cite{AdamczewskiKonieczny-2022} we defined the operations of addition, multiplication and the integer part for parametric {\gp} maps, not necessarily indexed by the same set. Roughly speaking, if $I \subset J$ are finite sets then we can always think of a {\gp} map parametrised by $\RR^I$ as a {\gp} map parametrised by $\RR^J$, with trivial dependence on the parameters in $\RR^{J \setminus I}$. Thus, if $g_{\bullet}$ and $h_{\bullet}$ are {\gp} maps parametrised by $\RR^I$ and $\RR^J$ respectively, then we can think of both $g_{\bullet}$ and $h_{\bullet}$ as {\gp} maps parametrised by $\RR^{I \cup J}$, which gives us a natural way to define the (pointwise) sum and product $g_{\bullet} + h_{\bullet}$ and $g_{\bullet} \cdot h_{\bullet}$. We refer to \cite{AdamczewskiKonieczny-2022}  for a formal definition. This construction directly extends to {\gp} maps parametrised by $\RR^{I_{\mathrm{real}}} \times \ZZ^{I_{\mathrm{int}}} \times [0,1)^{I_{\mathrm{frac}}}$.

\renewcommand{\succeq}{\leadsto}
 
\begin{definition}\label{def:ind:succ}
	Let $g_\bullet$ and $h_\bullet$ be two {\gp} maps parametrised by $\RR^{I_{\mathrm{real}}} \times \ZZ^{I_{\mathrm{int}}} \times [0,1)^{I_{\mathrm{frac}}}$ and $\RR^{J_{\mathrm{real}}} \times \ZZ^{J_{\mathrm{int}}} \times [0,1)^{J_{\mathrm{frac}}}$ respectively.
	Then we say that $h_\bullet$ \emph{extends} $g_\bullet$, denoted\footnote{We use different notation $h_\bullet \succeq g_\bullet$ than in \cite{AdamczewskiKonieczny-2022} in order to avoid confusion with the symbol $\succ$ extensively used in Section \ref{sec:hardy}} 
	 $h_\bullet \succeq g_\bullet$, 
if there exists a {\gp} map 
	$\varphi \colon \RR^{I_{\mathrm{real}}} \times \RR^{I_{\mathrm{int}}} \times \RR^{I_{\mathrm{frac}}} \to \RR^{J_{\mathrm{real}}} \times \RR^{J_{\mathrm{int}}} \times \RR^{J_{\mathrm{frac}}}$ such that 
	\begin{itemize}
	\item \( \displaystyle \varphi\bra{\RR^{I_{\mathrm{real}}} \times \ZZ^{I_{\mathrm{int}}} \times [0,1)^{I_{\mathrm{frac}}}} \subset \RR^{I_{\mathrm{real}}} \times \ZZ^{I_{\mathrm{int}}} \times [0,1)^{I_{\mathrm{frac}}} \); and 
	\item \(  \displaystyle g_{\vec \a} = h_{\varphi(\vec\a)} \) for all \(  \displaystyle \vec\a \in \RR^{I_{\mathrm{real}}} \times \ZZ^{I_{\mathrm{int}}} \times [0,1)^{I_{\mathrm{frac}}}\). 
	\end{itemize}

\end{definition}

In \cite{AdamczewskiKonieczny-2022} we obtained a polynomial bound on the number of possible prefixes of a given {\gp} map parametrised by $[0,1)^I$.

\begin{theorem}[{\cite[Thm.\ 15.3]{AdamczewskiKonieczny-2022}}]\label{thm:AK-sword-comp-param}
	Let $g_\bullet:\ZZ \to \ZZ$ be a {\gp} map parametrised by $[0,1)^I$ for some finite set $I$.
	Then there exists a constant $C$ such that, as $N \to \infty$, we have
\begin{equation}\label{eq:469:1}
	\abs{ \set{g_{\vec\a}|_{[N]}}{\a \in [0,1)^I }} = O \bra{N^{C}}.
\end{equation}	
Above, the implicit constant depends only on $g_{\bullet}$.
\end{theorem}
 
Our next goal is to obtain a similar bound for the number of prefixes of a bounded {\gp} map parametrised by $\RR^I$. Even though we are ultimately interested in bounded {\gp} maps, Proposition \ref{prop:param-reduce} concerning unbounded {\gp} maps is more amenable to proof by structural induction. We will use the following induction scheme.

\begin{proposition}[{\cite[Prop.\ 13.9]{AdamczewskiKonieczny-2022}}]\label{prop:gen-poly-induction}
	Let $\cG$ be a family of {\pgp} maps from $\ZZ$ to $\ZZ$ with index sets contained in $\NN$. Suppose that $\cG$ has the following closure properties. 
\begin{enumerate}
\item\label{it:A-1} All {\gp} maps $\ZZ \to \ZZ$ belong to $\cG$.
\item\label{it:A-2} For every  $g_{\bullet}$ and $h_{\bullet} \in \cG$, it holds that $g_{\bullet}+h_{\bullet} \in \cG$ and $g_{\bullet} \cdot h_{\bullet} \in \cG$.
\item\label{it:A-3} For every $g_\bullet \in \cG$, $\cG$ contains all the {\pgp} maps $g'_\bullet \colon \ZZ \to \ZZ$ satisfying $g_\bullet \succeq g'_\bullet$.
\item\label{it:A-4} For every pair of disjoint finite sets $I \subset \NN$, $J \subset \NN$, and every sequence of {\pgp} maps $h^{(i)}_{\bullet} \in \cG$, $i \in I$, 
with index set $J$, $\cG$ contains the {\pgp} map $g_\bullet$ defined by
\[ g_{\vec \a,\vec\b}(n) = \ip{\sum_{i\in I} \a_i h^{(i)}_{\vec\b}(n)}\,, \qquad n \in \ZZ\,,\ \vec\a \in \RR^{I}\,,\ \vec\b \in \RR^J\,.\]
\end{enumerate}
Then $\cG$ contains all {\pgp} maps $\ZZ \to \ZZ$ with index sets contained in $\NN$.
\end{proposition}

\begin{proposition}\label{prop:param-reduce}
	Let $g_{\bullet} \colon \ZZ \to \ZZ$ be a {\gp} map parametrised by $\RR^I$ for a finite set $I$. Then there exist finite sets $J,K$ and a {\gp} map $\tilde g_\bullet \colon \ZZ \to \ZZ$ parametrised by $\ZZ^J \times [0,1)^K$ such that $\tilde g_\bullet \succeq g_\bullet$ and $\tilde g_\bullet$ takes the form
	\begin{align*}
		\tilde g_{a,\b} &= \sum_{j \in J} a_j h^{(j)}_{\b}, &&& a \in \ZZ^J\,,\ \b \in [0,1)^K.
	\end{align*}
	where for each $j \in J$, $h^{(j)}_{\bullet} \colon \ZZ \to \ZZ$ is a {\gp} map parametrised by $[0,1)^K$.
\end{proposition}
\begin{proof}
\begin{enumerate}[wide]
\item If $g \colon \ZZ \to \ZZ$ is a fixed {\gp} map (i.e., if $I = \emptyset$) then we can simply take $\tilde g = g$.
\item Suppose that the conclusion holds for $g_\bullet,h_\bullet \colon \ZZ \to \ZZ$, and let the corresponding extensions $\tilde g_{\bullet}$ and $\tilde h_\bullet$ be given by 
	\begin{align*}
		\tilde g_{a,\b} &= \sum_{j \in J} a_j h^{(j)}_{\b}, &&& a \in \ZZ^J\,,\ \b \in [0,1)^K \\
		\tilde h_{c,\delta} &= \sum_{l \in L} c_{l} h^{(l)}_{\delta}, &&& c \in \ZZ^L\,,\ \delta \in [0,1)^M.
	\end{align*}
We may freely assume that the index sets $J,K,L,M$ are pairwise disjoint. 
We will show that the conclusion also holds for $g_{\bullet} + h_{\bullet}$ and $g_{\bullet} \cdot h_{\bullet}$. In the case of $g_{\bullet} + h_{\bullet}$ it is enough to combine the sums representing $\tilde g_{a,\b} $ and $\tilde h_{c,\delta}$ into a single sum. In the case of $g_{\bullet} \cdot h_{\bullet}$, we take
	\begin{align*}
	\tilde f_{e,(\beta,\delta)} &= \sum_{j \in J,\ l \in L} e_{j,l} \bra{ h^{(j)}_{\b} \cdot h^{(l)}_{\delta}} &&& e \in \ZZ^{J \times L}\,,\ (\beta,\delta) \in [0,1)^{K\times M}.
	\end{align*}
Then $\tilde f$ has the required form and (taking $e_{j,l} = a_j c_l$) we see that $\tilde f_{\bullet} \succeq \tilde g_{\bullet} \cdot \tilde h_{\bullet} \succeq g_{\bullet} \cdot h_{\bullet}$. 
\item Suppose that the conclusion holds for $g_\bullet$ and that $g_{\bullet} \succeq g'_\bullet$. Then the conclusion also holds for $g'_{\bullet}$ because the relation of being an extension is transitive.
\item Suppose that $I \subset \NN$, $J \subset \NN$ are disjoint finite sets,  $h^{(i)}_{\bullet}$ are {\gp} maps parametrised by $\RR^J$ which satisfy the conclusion for each for $i \in I$, and $g_\bullet$ is the {\pgp} map defined by
	\begin{align*}
	g_{\vec \a,\vec\b}(n) &:= \ip{\sum_{i\in I} \a_i h^{(i)}_{\vec\b}(n)},&&& n \in \ZZ\,,\ \vec\a \in \RR^{I}\,,\ \vec\b \in \RR^J.
	\end{align*}
Let the extensions of $h^{(i)}$ be given by
	\begin{align*} \tilde h_{c,\delta}^{(i)} &= \sum_{l \in L} c_{l} f^{(i,l)}_{\delta}, &&& c \in \ZZ^L\,,\ \delta \in [0,1)^M.
	\end{align*}
(Note that we may without loss of generality assume use the same index sets $L$ and $M$ for each $i \in I$.) We will show that the conclusion is satisfied for $g_\bullet$. We observe that we have the equality \begin{align*}
		\ip{\sum_{i\in I} \a_i \tilde h_{c,\delta}^{(i)}} 
		&= \ip{\sum_{i\in I,\ l \in L}  \a_i c_{l} f^{(i,l)}_{\delta} } 
		 &= \sum_{i\in I,\ l \in L}  \ip{\a_i c_{l}} f^{(i,l)}_{\delta} 
		+ \ip{\sum_{i\in I,\ l \in L}  \fp{\a_i c_{l}} f^{(i,l)}_{\delta} } 
	\end{align*}
	This motivates us to define
	\begin{align*}
	\tilde g_{e,\delta,\phi} &:= \sum_{i \in I,\ l \in L} e_{i,l} f^{(i,l)}_{\delta} 
	+ e_{\diamond} \ip{\sum_{i\in I,\ l \in L} \phi_{i,l} f^{(i,l)}_{\delta} } 
	\\ &
	e \in \ZZ^{I \times L \cup \{\diamond\}}, \phi \in [0,1)^{I \times L}, \delta \in [0,1)^M,
	\end{align*}
	where $\diamond$ is some index that does not belong to $I \times J$.
Letting also 
	\begin{align*}
		f^{(\diamond)}_{\delta,\phi} := 
	\ip{\sum_{i\in I,\ l \in L} \phi_{i,l} f^{(i,l)}_{\delta} } 
	&&& \phi \in [0,1)^{I \times L}\,,\ \delta \in [0,1)^M,
	\end{align*}
	we see that $\tilde g_{\bullet}$ takes the required form and (setting $\phi_{i,l} = \fp{\a_i c_l}$ and $e_{\diamond} = 1$) we have $\tilde g_\bullet \succeq g_\bullet$.
\end{enumerate}
Combining the closure properties proved above, we infer from Proposition \ref{prop:gen-poly-induction} that the conclusion holds for all {\pgp} maps.
\end{proof}

\begin{proposition}\label{prop:param-reduce-bdd}
	Let $M \in \NN$ and let $g_{\bullet} \colon \ZZ \to [M]$ be a {\gp} map parametrised by $\RR^I$ for a finite set $I$. Then there exist a  {\gp} map $\tilde g_\bullet \colon \ZZ \to [M]$ parametrised by $[0,1)^J$ for a finite set $J$ such that $\tilde g_\bullet \succeq g_\bullet$.
\end{proposition}
\begin{proof}
	Let $\tilde g^{(0)}_{\bullet} \succeq g_{\bullet}$ be the {\pgp} from Proposition \ref{prop:param-reduce}, and let
	\begin{align*}
		\tilde g_{a,\b}^{(0)} &= \sum_{j \in J} a_j h^{(j)}_{\b}, &&& a \in \ZZ^J, \b \in [0,1)^K.
	\end{align*}
Since the value of $g_{\a,\b}(n)$ is completely determined by its residue modulo $M$, we expect that it is enough to consider the values of $a$ with $a \in [M]^J$.
This motivates us to put
	\begin{align*}
		\tilde g_{\a,\b} &= \sum_{j \in J} \ip{M \a_j} h^{(j)}_{\b}, &&& \a \in [0,1)^J, \b \in [0,1)^K.
	\end{align*}
Let $\phi \colon \ZZ^I \to \ZZ^J$ and $\psi \colon \ZZ^I \to \RR^K$ be {\gp} maps such that $g_{\a} = \tilde g^{(0)}_{\phi(\a),\psi(\a)}$. Let $\theta \colon \ZZ^I \to [0,1)^J$ be given by $\theta(\a) := \fp{\phi(\a)/M }$ (with fractional part taken coordinatewise). Then 
	\begin{align*}
		\tilde g_{\phi(\a),\b}^{(0)}(n) \equiv \tilde g_{\theta(\a),\b}(n) \bmod{M}
		, &&& \text{ for all } n \in \ZZ, \a \in \RR^I, \b \in [0,1)^K.
	\end{align*}
	Since $g_{\bullet}$ takes values in $[M]$, it follows that
	\begin{align*}
	g_{\alpha}(n) =	\tilde g_{\phi(\a),\psi(\a)}^{(0)}(n) \equiv \tilde g_{\theta(\a),\psi(\a)}(n)  \bmod {M}
		, &&& \text{ for all } n \in \ZZ, \a \in \RR^I.
	\end{align*}
Replacing $\tilde g_\bullet$ with $M \cdot \fp{\tilde{g_\bullet}/M}$ if necessary, we may further ensure that $\tilde g_\bullet$ takes values in $[M]$. As a consequence, $\tilde g_\bullet \succeq g_\bullet$, as needed. 
\end{proof}

\begin{proposition}
	Let $\bb a = (a(n))_{n \in \ZZ}$ be a (two-sided) bracket word over a finite alphabet $\Sigma$, and let $g_{\bullet} \colon \ZZ \to \ZZ$ be a {\gp} map parametrised by $\RR^I$ for some finite set $I$. Then there exists a constant $C > 0$ such that, as $N \to \infty$, we have
\[
	\abs{\set{\brabig{ a\bra{g_{\a} (n)}  }_{n=0}^{N-1}}{\a \in \RR^I} } = O(N^C).
\]
Above, the implicit constant depends on $\bb a$ and $g_\bullet$.
\end{proposition}
\begin{proof}
	Let $M := \abs{\Sigma}$. We may freely assume that $\Sigma = [M]$, in which case $a$ is a {\gp} map. Thus, $a\circ g_{\bullet}$ is a {\gp} map parametrised by $\RR^I$ and taking values in $[M]$. By Proposition \ref{prop:param-reduce-bdd}, there exists a {\gp} map $\tilde g_{\bullet}$ parametrised by $[0,1)^J$  for a finite set $J$ such that $\tilde g_{\bullet} \succeq a \circ g_{\bullet}$. Thus, it suffices to show that, for a certain $C > 0$, the number of words  $\displaystyle \bra{\tilde g_{\a}(n) }_{n=0}^{N-1}$ for $\a \in [0,1)^J$ is $O(N^C)$ as $N \to \infty$. This is precisely Theorem  \ref{thm:AK-sword-comp-param}.
\end{proof}

As a special case, we obtain a bound on the number of subsequences of bracket words along polynomials of a given degree.

\begin{corollary}\label{cor:bracket_word_along_polys}
	Let $\bb a = (a(n))_{n \in \ZZ}$ be a (two-sided) bracket word over a finite alphabet $\Sigma$ and let $d \in \NN$. Then there exists a constant $C > 0$ such that, as $N\to \infty$ we have
	\begin{align*}
		\abs{\set{(a(\ip{p(n)}))_{n=0}^{N-1}}{p \in \R_{\leq d}[x]}} = O(N^C),
	\end{align*}
	where the implied constant depends only on $\bb a$ and $d$.

\end{corollary}

Thus we are now in a position to prove Theorem~\ref{thm:A}.
\begin{proof}[Proof of Theorem~\ref{thm:A}]
	We aim to estimate the number of subwords of length $H$ of $(a(\ip{f(n)}))_{n=0}^{\infty}$, that is, we count words of the form 
	\begin{align*}
(a(\ip{f(N)}), \ldots, a(\ip{f(N+H-1)})) = (a(\ip{f(N+h)}))_{h=0}^{H-1}
	\end{align*}
	for $N\in \N$.
	Since $f$ has polynomial growth, there exists $k \in \N$ such that $f(t) \ll t^k$. We choose $\ell \geq k+1$ and apply Theorem~\ref{thm:error} to find some $0<\eta<1$ such that for any $H \in \N$ at least one of the following holds
\begin{enumerate}
\item\label{it:error:A_2} \textit{$N$ is small:} $N = O(H^{\bra{\ell + \eta}/\bra{\ell - k}})$. 
\item\label{it:error:B_2} \textit{$e_N$ is sparse:} There are at most $O(H^{\eta})$ values of $h \in [H]$ such that $e_N(h) \neq 0$.
\item\label{it:error:C_2} \textit{$e_N$ is structured:} There exists a partition of $[H]$ into $O(H^{\eta})$ arithmetic progressions with step $O(H^{\eta})$ on which $e_N$ is constant,
\end{enumerate}
where 
	\begin{align*}
		e_{N}(h) &:= \floor{f(N+h)} - \floor{P_{N,\ell}(h)}, &&& 0 \leq h < H
	\end{align*}
and $P_{N,\ell}$ is the Taylor polynomial of $f$ (see \eqref{eq:Taylor-2}).
We distinguish the three possible cases. Obviously \ref{it:error:A_2} contributes at most $O(H^{\ell + 1})$ different words.
For \ref{it:error:B_2} we first consider $a(\ip{P_{N,\ell}(h)})_{h=0}^{H-1}$. By Corollary~\ref{cor:bracket_word_along_polys} this word is contained in a set of size $O(H^C)$.
By assumption $a(\ip{f(N+h)}) \neq a(\ip{P_{N,\ell}(h)})$ for at most $O(H^{\eta})$ values of $h\in [H]$, which can be chosen in $\binom{H}{O(H^{\eta})}$ ways
For each position $h$ with $a(\ip{f(N+h)}) \neq a(\ip{P_{N,\ell}(h)})$ we have at most $\abs{\Sigma}$ possibilities for the value of $a(\ip{f(N+h)})$.
In total, we can estimate the number of subwords of length $H$ in this case (up to a constant) by
\begin{align*}
	H^C \cdot \binom{H}{O(H^{\eta})} \cdot \abs{\Sigma}^{O(H^{\eta})} &\leq H^C \cdot H^{O(H^{\eta})} \cdot \abs{\Sigma}^{O(H^{\eta})}\\
		&= \exp\rb{C \log H + O((\log H) \cdot H^{\eta}) + O((\log \abs{\Sigma}) \cdot H^{\eta})}\\
		&= \exp\rb{O_{C, \eta}(H^{(1+\eta)/2})}. 
\end{align*}
In the last case~\ref{it:error:C_2} we decompose $[H]$ into $O(H^{\eta})$ arithmetic progressions on which $e_N$ is constant. We let these arithmetic progressions be denoted by $P_1, \ldots, P_s$. As there are at most $H^3$ arithmetic progressions contained in $[H]$ we can bound the number of possible different decompositions by $(H^3)^{O(H^{\eta})}$.
On every such progression there exists a polynomial $q$ (which is either $P_{N,\ell}, P_{N,\ell}+1$ or $P_{N,\ell}-1$) such that $a(\ip{f(N+h)}) = a(\ip{q(h)})$. As a polynomial along an arithmetic progression is again a polynomial, by Corollary~\ref{cor:bracket_word_along_polys} we can bound the number of subwords appearing along some $P_j$ by $H^C$.
In total, we can estimate the number of subwords of length $H$ in this case by
\begin{align*}
	(H^3)^{O(H^{\eta})}\cdot (H^C)^{O(H^{\eta})} &= 
	\exp( (C+3) \log(H) \cdot O(H^{\eta}))\\
		&= \exp(O_{C,\eta}(H^{(1+\eta)/2})).
\end{align*}
This finishes the proof for $\delta = (1+\eta)/2 < 1$.
\end{proof}

\section{Nilmanifolds}\label{sec:nilmanifolds}

In this section we we recall some basic definitions and results on nilmanifolds and discuss the connection to generalized polynomials which goes back to the work of Bergelson and Leibman \cite{BergelsonLeibman-2007}. 

\subsection{Basic definitions}
In this section, we very briefly introduce definitions and basic facts related to nilmanifolds and nilpotent dynamics. Throughout this section, we let $G$ denote an $s$-step nilpotent Lie group of some dimension $D$. We assume that $G$ is connected and simply connected. We also let $\Gamma < G$ denote a subgroup that is discrete and cocompact, meaning that the quotient space $G/\Gamma$ is compact. The space $X = G/\Gamma$ is called a \emph{$s$-step nilmanifold}. A \emph{degree-$d$ filtration} on $G$ is a sequence $G_\bullet$ of subgroups 
\[
	G = G_0 = G_1 \geq G_2 \geq G_3 \geq \dots
\] such that $G_{d+1} = \{e_G\}$ (and hence $G_{i} = \{e_G\}$ for all $i > d$) and for each $i,j$ we have $[G_i,G_j] \subset G_{i+j}$, where $[G_i,G_j]$ is the group generated by the commutators $[g,h] = ghg^{-1}h^{-1}$ with $g \in G_i$, $h \in G_j$. A standard example of a filtration is the \emph{lower central series} given by $G_{(0)} = G_{(1)} = G$ and $G_{(i+1)} = [G,G_{(i)}]$ for $i \geq 1$.

A \emph{Mal'cev basis} compatible with $\Gamma$ and $G_\bullet$ is a basis $\cX = (X_1,X_2,\dots,X_D)$ of the Lie algebra $\mathfrak{g}$ of $G$ such that
\begin{enumerate}
\item for each $0 \leq j \leq D$, the subspace $\fh_j := \operatorname{span}\bra{X_{j+1},X_{j+2},\dots,X_D}$ is a Lie algebra ideal in $\fg$;
 \item for each $0 \leq i \leq d$, each $g \in G_i$ has a unique representation as $g = \exp(t_{D(i) + 1} X_{t_{D(i) + 1}})  \cdots \exp(t_{D-1} X_{D-1}) \exp(t_D X_D) $, where $D(i) := \operatorname{codim} G_i$ and $t_j \in \RR$ for $D(i) < j \leq D$;
 \item $\Gamma$ is the set of all products $\exp(t_1 X_1) \exp(t_2 X_2) \cdots \exp(t_D X_D)$ with $t_j \in \ZZ$ for $1 \leq j \leq D$.
\end{enumerate}
If the Lie bracket is given in coordinates by
\[
	[X_i,X_j] = \sum_{k=1}^D c^{(k)}_{i,j} X_k,
\]
where all of the constants $c^{(k)}_{i,j}$ are rationals with height at most $M$ then we will say that the complexity of $(G,\Gamma,G_\bullet)$ is at most $M$. We recall that the height of a rational number $a/b$ is $\max(\abs{a},\abs{b})$ ($a \in \ZZ$, $b \in \NN$, $\gcd(a,b) = 1$).

We will usually keep the the choice of the Mal'cev basis implicit, and assume that each filtered nilmanifold under consideration comes equipped with a fixed choice of Mal'cev basis. The Mal'cev basis $\cX$ induces coordinate maps $\tau \colon X \to [0,1)^D$ and $\tilde \tau \colon G \to \RR^D$, such that
\begin{align*}
	x &= \exp(\tau_1(x)X_1) \exp(\tau_2(x)X_2) \cdots \exp(\tau_D(x)X_D) \Gamma, &&& x \in X \\
	g &= \exp(\tilde\tau_1(g)X_1) \exp(\tilde\tau_2(g)X_2) \cdots \exp(\tilde\tau_D(g)X_D), &&& g \in G.
\end{align*}
The Mal'cev basis also induces a natural choice of a right-invariant metric on $G$ and a metric on $X$. We refer to \cite[Def.\ 2.2]{GreenTao-2012} for a precise definition. Keeping the dependence on $\cX$ implicit, we will use the symbol $d$ to denote either of those metrics.

The space $X$ comes equipped with the Haar measure $\mu_X$, which is the unique Borel probability measure on $X$ invariant under the action of $G$: $\mu_X(gE) = \mu_X(E)$ for all measurable $E \subset X$ and $g \in G$. When there is no risk of confusion, we write $dx$ as a shorthand for $d\mu_X(x)$.

A map $g \colon \ZZ \to G$ is polynomial with respect to the filtration $G_\bullet$, denoted $g \in \mathrm{poly}(\ZZ,G_\bullet)$, if it takes the form
\[
	g(n) = g_0^{} g_1^{n} \dots g_d^{\binom{n}{d}}, 
\]
where $g_i \in G_i$ for all $0 \leq i \leq d$ (cf.\ \cite[Lem.\ 6.7]{GreenTao-2012}; see also \cite[Def.\ 1.8]{GreenTao-2012} for an alternative definition). Although it is not immediately apparent from the definition above, polynomial sequences with respect to a given filtration form a group and are preserved under dilation. 

\subsection{Semialgebraic geometry}
A basic semialgebraic set $S \subset \RR^D$ is a set given by a finite number of polynomial equalities and inequalities:
\begin{equation}\label{eq:Nil:def-S}
	S = \set{ x \in \RR^d }{
		P_1(x) > 0,\dots,P_n(x) > 0, 
		Q_1(x) = 0,\dots,Q_m(x) = 0
	}.
\end{equation}
A semialgebraic set is a finite union of basic semialgebraic sets. In a somewhat ad hoc manner, we define the complexity of the basic semialgebraic set $S$ given by \eqref{eq:Nil:def-S} to be the sum $\sum_{i=1}^n \deg P_i + \sum_{j=1}^m \deg Q_j$ of degrees of polynomials appearing in its definition. (Strictly speaking, we take the infimum over all representations of $S$ in the form \eqref{eq:Nil:def-S}.) We also define the complexity of a semialgebraic set 
\begin{equation}\label{eq:Nil:def-S-union}
	S = S_1 \cup S_2 \cup \dots \cup S_r.
\end{equation}
represented to be the finite union of basic semiaglebraic sets $S_i$ as the sum of complexities of $S_i$. (Again, we take the infimum over all representations \eqref{eq:Nil:def-S-union}.)

Using the Mal'cev coordinates to identify the nilmanifold $X$ with $[0,1)^D$, we extend the notion of a semialgebraic set to subsets of $X$.  A map $F \colon X \to \RR$ is piecewise polynomial if there exists a partition $X = \bigcup_{i=1}^r S_i$ into semialgebraic pieces and polynomial maps $\Phi_i \colon \RR^D \to \RR$ such that $F(x) = \Phi_i(\tau(x))$ for each $1 \leq i \leq r$ and $x \in S_i$. One can check that these notions are independent of the choice of basis, although strictly speaking we will not need this fact.

\subsection{Quantitative equidistribution}

The Lipschitz norm of a function $F \colon X \to \RR$ is defined as
\[
	\normLip{F} = \norm{F}_{\infty} + \sup_{x,y \in X,\ x \neq y} \frac{\abs{F(x)-F(y)}}{d(x,y)}
\]
A sequence $(x_n)_{n=0}^{N-1}$ in $X$ is \emph{$\delta$-equidistributed} if for each Lipschitz function $F \colon X \to \RR$ we have
\[
	\abs{ \EEE_{n < N} F(x_n) - \int_X F(x) dx} \leq \delta \normLip{F}.
\]
In the case, where $X = [0,1]$ this notion is highly connected to the discrepancy of a sequence (see~\eqref{eq_discrepancy}). In fact, for $\delta>0$ small enough we have that $(x_n)_{n=0}^{N-1}$ has discrepancy $\delta$ if and only if it is $\delta^{O(1)}$ distributed. One direction follows immediately from the Koksma-Hlawka inequality and the other direction can be found for example in the proof of Proposition 5.2 in~\cite{Deshouillers2022}.

More restrictively, $(x_n)_{n=0}^{N-1}$ is \emph{totally $\delta$-equidistributed} if for each arithmetic progression $P \subset [N]$ of length at least $\delta N$ we have
\[
	\abs{ \EEE_{n \in P} F(x_n) - \int_X F(x) dx} \leq \delta \normLip{F}.
\]
A sequence $(\e_n)_{n=0}^{N-1}$ in $G$ is \emph{$(M,N)$-smooth} if $d(\e_n, e_G) \leq M$ and $d(\e_n,\e_{n+1}) \leq M/N$ for all $n \in [N-1]$.
A group element $\gamma \in G$ is \emph{$Q$-rational} if $\gamma^r \in \Gamma$ for some positive integer $r \leq Q$. A point $x \in G/\Gamma$
is \emph{$Q$-rational} if it takes the form $x = \gamma\Gamma$ for some $Q$-rational $\gamma \in G$. A sequence $(x_n)_{n=0}^{N-1}$ in $X$ is $Q$-rational if each point $x_n$ is $Q$-rational. 

\begin{theorem}[{\cite[Thm.\ 1.19]{GreenTao-2012}}]\label{thm:GT-factor}
	Let $C > 0$ be a constant.
	Let $G$ be a connected, simply connected nilpotent Lie group of dimension $D$, let $\Gamma < G$ be a lattice, let $G_\bullet$ be a nilpotent filtration on $G$ of length $d$, and assume that the complexity of $(G,\Gamma,G_\bullet)$ is at most $M_0$. Then for each $N \in \NN$ and each polynomial sequence $g \in \mathrm{poly}(\ZZ,G_\bullet)$ there exists an integer $M$ with $M_0 \leq M \ll M_0^{O_{C,d,D}(1)}$ and a decomposition 
$g(n) = \e(n) g'(n) \gamma(n)$ ($n \in \ZZ$), where $\e,g',\gamma \in \mathrm{poly}(\ZZ,G_\bullet)$ and 
\begin{enumerate}
\item the sequence $\bra{\e(n)}_{n=0}^{N-1}$ is $(M,N)$-smooth;
\item the sequence $\bra{\gamma(n)\Gamma}_{n=0}^{N-1}$ is $M$-rational and periodic with period $\leq M$;
\item there is a group $G' < G$ with Mal'cev basis $\cX'$ in which each element is an $M$-rational combination of elements of $\cX$ such that $g'(n) \in G'$ for all $n \in \ZZ$, and the sequence  $\bra{g'(n)\Gamma'}_{n=0}^{N-1}$ is totally $1/M^C$-equidistributed in $G'/\Gamma'$, where $\Gamma' = \Gamma \cap G'$.
\end{enumerate}
\end{theorem}

\subsection{Generalised polynomials}

The connection between nilmanifolds and generalised polynomials was first elucidated by Bergelson and Leibman \cite{BergelsonLeibman-2007}.

\begin{theorem}[{\cite{BergelsonLeibman-2007}}]\label{thm:BL}
	Let $f \colon \ZZ \to [0,1)$ be a sequence. Then the following conditions are equivalent:
\begin{enumerate}
\item\label{it:BL-1} $f$ is a {\gp} map;
\item\label{it:BL-2} there exists a connected, simply connected nilpotent Lie group $G$, lattice $\Gamma < G$, $g \in G$ and a piecewise polynomial map $F \colon G/\Gamma \to [0,1)$ such that $f(n) = F(g^n\Gamma)$ for all $n \in \ZZ$;
\item\label{it:BL-3} there exists a connected, simply connected nilpotent Lie group $G$ of some dimension $D$, lattice $\Gamma < G$, a compatible filtration $G_\bullet$, a polynomial sequence $g \in \mathrm{poly}(\ZZ,G_\bullet)$ and an index $1 \leq j \leq D$ such that $f(n) = \tau_j(g(n)\Gamma)$ for all $n \in \ZZ$. 
\end{enumerate}
\end{theorem}
\begin{remark}
	Strictly speaking, \cite{BergelsonLeibman-2007} does not include the assumption that $G$ should be connected and simply connected. However, this requirement can be ensured by replacing $G$ with a larger group. (cf. the ``lifting argument'' on \cite[p.\ 368]{Frantzikinakis-2009} and also \cite[Thm.\ A*]{BergelsonLeibman-2007}). The cost of this operation is that in \ref{it:BL-2} one may not assume that the action of $g$ on $G/\Gamma$ is minimal, but we do not need this assumption. 
\end{remark}

In our applications, we will need to simultaneously represent maps of the form $f(\ip{p(n)})$ where $f$ is a fixed {\gp} map and $p$ is a polynomial which is allowed to vary. Such a representation is readily obtained from Theorem \ref{thm:BL}.

\begin{theorem}\label{thm:gp-param}
	Let $f \colon \ZZ \to \RR$ be a bounded {\gp} map and let $d \in \NN$. Then there exists a connected, simply connected nilpotent Lie group $G$, a lattice $\Gamma < G$, a filtration $G_\bullet$, and a piecewise polynomial map $F \colon G/\Gamma \to \ZZ$ such that for each polynomial $p(x) \in \RR[x]$ with $\deg p \leq d$ there exists $g_p \in \mathrm{poly}(G_\bullet)$ such that for all $n \in \ZZ$ we have $f\bra{\ip{p(n)}} = F(g_p(n)\Gamma)$. 
\end{theorem}
\begin{proof}
	By Theorem \ref{thm:BL}, there exists a nilmanifold $G^{(0)}/\Gamma^{(0)}$ together with a piecewise polynomial map $F^{(0)} \colon G^{(0)}/\Gamma^{(0)} \to \RR$, and a group element $g_0 \in G^{(0)}$ such that $f(n) = F^{(0)}(g_0^n \Gamma)$ for all $n \in \ZZ$. 
	Following the strategy in \cite[Lem.\ 4.1]{Frantzikinakis-2009}, let $G :=   G^{(0)} \times \RR$ and $\Gamma :=   \Gamma^{(0)} \times \ZZ$ and let $F \colon G/\Gamma \to \RR$ be given by $F(t + \ZZ,h   \Gamma^{(0)}) :=   F^{(0)}(   g_0^{-\fp{t}} h    \Gamma^{(0)})$ for $t \in \RR$ and $h \in   G^{(0)}$. This construction guarantees that $F$ is piecewise polynomial and for all $t \in \RR$ we have
	\[ F(t+\ZZ,  g_0^t \Gamma) = F^{(0)}(  g_0^{\ip{t}} \Gamma ) = f(\ip{t}). \]
	For $p \in \RR[x]$ and $n \in \ZZ$ let
\(
	g_p(n) := \bra{p(n),  g_0^{p(n)}}
\).
Then $g_\a$ is polynomial with respect to the filtration $G_\bullet$ given by $G_{i} = G_{(\ip{i/d})}$, where $\bra{ G_{(j)} }_j$ denotes the lower central series, and we have $f\bra{\ip{p(n)}} = F(g_p(n)\Gamma)$ for all $n \in \ZZ$.
\end{proof} 
\newcommand{\hh}{k}	

\section{M\"{o}bius orthogonality}
\label{sec:mobius} 

\subsection{Main result}
In this section, we discuss M\"obius orthogonality of bracket words along Hardy field sequences. Our main result is Theorem \ref{thm:B}, which we restate below.

\begin{theorem}\label{thm:Mobius}
	Let ${\bb a} = (a(n))_{n \in \ZZ}$ be a (two-sided) $\RR$-valued bracket word and let $f \colon \RR_+ \to \RR$ be a Hardy field function with polynomial growth. Then
	\begin{align}\label{eq:Mobius:main-0}
	&&&& 	\frac{1}{N} \sum_{n=1}^N \mu(n) a\bra{\ip{f(n)}}  &\to 0 &&& \text{as } N \to \infty. 
	\end{align}
\end{theorem}

As usual, we will use Taylor expansion to approximate the restriction of $f(n)$ to an interval with a polynomial sequence, and then use Theorem \ref{thm:error} to control the error term involved in computing $\ip{f(n)}$. The sequence $a(\ip{f(n)})$ can then be represented on a nilmanifold by Bergelson--Leibman machinery. As the next step, we require a suitable result on M\"obius orthogonality in short intervals. In Section \ref{sec:short-int}, we will prove the following theorem, which is closely related to \cite[Thm.\ 1.1(i)]{MSTT}.
Below, we let $\mathcal{AP}$ denote the set of all arithmetic progressions in $\ZZ$.

\begin{theorem}\label{thm:Mobius:semialg}
	Let $G$ be a connected, simply connected nilpotent Lie group, let $\Gamma < G$ be a lattice, let $G_\bullet$ be a filtration on $G$, assume that $G_\bullet$ and $\Gamma$ are compatible, and let $F \colon G/\Gamma \to \RR$ be finitely-valued piecewise polynomial map. Let $N,H$ be integers with $N^{0.626} \leq H \leq N$. Then
	\begin{align}\label{eq:Mobius:semialg}
 \sup_{g \in \mathrm{poly}(G_\bullet)} \sup_{P \in \mathcal{AP}} 
 	\abs{ \EEE_{h < H} 
		1_{P}(h) \mu(N+h) F(g(h)\Gamma)
	} = o_{N\to\infty}(1),
\end{align}
where the rate of convergence may depend on $G,\Gamma,G_\bullet$ and $F$.
\end{theorem}

\begin{proof}[Proof of Theorem \ref{thm:Mobius} assuming Theorem \ref{thm:Mobius:semialg}]
Applying a dyadic decomposition, it will suffice to show that
\begin{align}\label{eq:Mobius:67:1}
\EEE_{N \leq n < 2N}  \mu(n) a\bra{\ip{f(n)}} &\to 0 &&& \text{as } N\to \infty.
\end{align}
Fix a small $\e > 0$. We will show that, for all sufficiently large $N$ we have
\begin{align}\label{eq:Mobius:67:2}
\abs{ \EEE_{N \leq n < 2N} \mu(n) a\bra{\ip{f(n)}}} \ll \e.
\end{align}
Splitting the average in \eqref{eq:Mobius:67:2} into intervals of length $\ceil{(2N)^{0.7}}$, we see that \eqref{eq:Mobius:67:2} will follow once we show that for sufficiently large $N$ and for $H$ satisfying $N^{0.7} \leq H < N$ we have
\begin{align}\label{eq:Mobius:67:3}
\abs{ \EEE_{h < H} \mu(N+h) a\bra{\ip{f(N+h)}}} \ll \e.
\end{align}
Pick an integer $k \in \NN$ such that $f(t) \ll t^k$, and let $\ell = 10 k$. By Theorem \ref{thm:error}, we have
\begin{align}\label{eq:Mobius:67:4}
	\ip{f(N+h)} = \ip{P_{N}(h)} + e_N(h), 
\end{align}
where $P_N$ is a polynomial of degree (at most) $\ell$ and one of the conditions \ref{thm:error}\ref{it:error:A}--\ref{it:error:C} holds. In the case \ref{it:error:A} we have $N \ll_{\e} H^{10/9} \leq N^{7/9}$, which implies that $N = O_\e(1)$. Assuming that $N$ is sufficiently large, we may disregard this case.

In the case \ref{it:error:B} we have $\EEE_{h < H} \abs{e_N(h)} < \e$, and as a consequence 
\begin{align}\label{eq:Mobius:67:5}
 \EEE_{h < H} \mu(N+h) a\bra{\ip{f(N+h)}} = 
  \EEE_{h < H} \mu(N+h) a\bra{\ip{P_N(h)}} + O(\e).
\end{align}
By Theorem \ref{thm:gp-param}, there exists a connected and simply connected nilpotent Lie group $G$, a lattice $\Gamma < G$, a filtration $G_\bullet$ and a finitely-valued piecewise polynomial map $F \colon G/\Gamma \to \ZZ$ such that for each polynomial $P$ of degree at most $\ell$ there exists $g \in \mathrm{poly}(G_\bullet)$ such that $a(\ip{P(h)}) = F(g(h)\Gamma)$. In particular,
\begin{align}\label{eq:Mobius:67:6}
\abs{ \EEE_{h < H} \mu(N+h) a\bra{\ip{P_N(h)}} }
\leq \sup_{g \in \mathrm{poly}(G_\bullet)}\abs{ \EEE_{h < H} \mu(N+h) a\bra{F(g(h)\Gamma} }
\end{align}
By Theorem \ref{thm:Mobius:semialg}, for sufficiently large $N$ the expression in \eqref{eq:Mobius:67:6} is bounded by $\e$. Inserting this bound into \eqref{eq:Mobius:67:5} yields \eqref{eq:Mobius:67:3}.

In the case \ref{it:error:C}, passing to an arithmetic progression we may replace $e_N$ with a constant sequence:
\begin{align}\label{eq:Mobius:67:7}
 &\abs{ \EEE_{h < H} \mu(N+h) a\bra{\ip{f(N+h)}} } 
 \\ \ll_{\e} & 
 \max_{P \in \mathcal{AP}} \max_{e \in \{-1,0,1\} } \abs{ \EEE_{h < H} \mu(N+h) 1_{P}(h) a\bra{\ip{P_N(h)}+e}}.
\end{align}
To finish the argument, it suffices to apply Theorem \ref{thm:Mobius:semialg} similarly to the previous case. \end{proof}

\subsection{Short intervals}
\label{sec:short-int} 
 
The remainder of this section is devoted to proving Theorem \ref{thm:Mobius:semialg}. We will derive it from closely related estimates for correlations of the M\"obius function with nilsequences in short intervals. Recall that we let $\mathcal{AP}$ denote the set of all arithmetic progressions in $\ZZ$.

\begin{theorem}[{Corollary of Thm.\ 1.1(i) in \cite{MSTT}}]\label{thm:Mobius:MSTT}
	Let $N,H$ be integers with $N^{0.626} \leq H \leq N$ and let $\delta \in (0,1/2)$. Let $G$ be a connected, simply connected nilpotent Lie group of dimension $D$, let $\Gamma < G$ be a lattice, let $G_\bullet$ be a nilpotent filtration on $G$ of length $d$, and assume that the complexity of $(G,\Gamma,G_\bullet)$ is at most $1/\delta$. Let $F \colon G/\Gamma \to \CC$ be a function with Lipschitz norm at most $1/\delta$. Then, for each $A > 0$ we have the bound
	\begin{align}\label{eq:Mobius:MSTT}
 \sup_{g \in \mathrm{poly}(G_\bullet)}  \sup_{P \in \mathcal{AP}}  \abs{ \EEE_{h < H} 
		 \mu(N+h) 1_P(h)F(g(h)\Gamma)
	} \ll_{A} \frac{ (1/\delta)^{O_{d,D}(1)}}{\log^A N}.
\end{align}
\end{theorem} 
 
This theorem is almost the ingredient that we need, except that in our application the function $F$ is not necessarily continuous (much less Lipschitz). Instead, $F$ is a finitely-valued piecewise polynomial  function, meaning that there exists a partition $G/\Gamma = \bigcup_{i=1}^r S_i$ into semialgebraic pieces and constants $c_i \in \RR$ such that for each $x \in X$ and $1 \leq i \leq r$, $F(x) = c_i$ if and only if $x \in S_i$. In this case, it is enough to consider each of the level sets separately. It is clear that Theorem \ref{thm:Mobius:semialg} will follow from the following more precise result.
 
\begin{theorem}\label{thm:Mobius:1_S-a}
	Let $N,H$ be integers with $N^{0.626} \leq H \leq N$ and let $\delta \in (0,1/2)$. Let $G$ be a connected, simply connected nilpotent Lie group of dimension $D$, let $\Gamma < G$ be a lattice, let $G_\bullet$ be a nilpotent filtration on $G$ of length $d$, and assume that the complexity of $(G,\Gamma,G_\bullet)$ is at most $1/\delta$. Let $S \subset G/\Gamma$ be a semialgebraic set with complexity at most $E$. Then, for each $A \geq 1$ we have the bound
	\begin{align}\label{eq:Mobius:1_S-a}
 \sup_{g \in \mathrm{poly}(G_\bullet)}  \sup_{P \in \mathcal{AP}} 
 	\abs{ \EEE_{h < H} 
		 \mu(N+h) 1_P(h)1_S(g(h)\Gamma)
	} \ll_{A} \frac{ (1/\delta)^{O_{d,D,E}(1)}}{\log^A N}.
\end{align}
\end{theorem} 
 
In the case where $(g(n)\Gamma)_n$ is highly equidistributed in $G/\Gamma$, we will derive Theorem \ref{thm:Mobius:1_S-a} directly from Theorem \ref{thm:Mobius:MSTT}. In fact, we will obtain a slightly stronger version, given in Theorem \ref{thm:Mobius:1_S-b} below. Then, we will deduce the general case of Theorem \ref{thm:Mobius:1_S-a} using the factorisation theorem from \cite{GreenTao-2012}. In order to avoid unnecessarily obfuscating the notation, from this point onwards we will allow all implicit constants to depend on the parameters $d$, $D$ and $E$; thus, for instance, the term on the right-hand side of \eqref{eq:Mobius:1_S-a} will be more succinctly written as ${ (1/\delta)^{O(1)}}/{\log^A N}$.

\subsection{Equidistributed case}

\begin{proposition}\label{thm:Mobius:1_S-b}
	Let $N,H$ be integers with $N^{0.626} \leq H \leq N$ and let $\delta \in (0,1/2)$. Let $G$ be a connected, simply connected nilpotent Lie group of dimension $D$, let $\Gamma < G$ be a lattice, let $G_\bullet$ be a nilpotent filtration on $G$ of length $d$, and assume that the complexity of $(G,\Gamma,G_\bullet)$ is at most $1/\delta$. Let $S \subset (\RR/\ZZ) \times (G/\Gamma)$ be a semialgebraic set with complexity at most $E$. Then, for each $A \geq 1$ there exists $B = O(A)$ such that 
	\begin{align}\label{eq:Mobius:1_S-b}
 \sup_{\substack{g \in \mathrm{poly}(G_\bullet)\\ \tilde\delta-\mathrm{t.e.d.}}}  \sup_{P \in \mathcal{AP}} 
 	\abs{ \EEE_{h < H} 
		 \mu(N+h) 1_P(h)1_S\bra{\frac{h}{H},g(h)\Gamma}
	} \ll_{A} \frac{ (1/\delta)^{O(1)}}{\log^A N},
\end{align}
where $\tilde\delta:= 1/\log^B{N}$ and the supremum is taken over all polynomial sequences $g$ such that $\bra{g(h)\Gamma}_{h =0}^H$ is totally $\tilde\delta$-equidistributed.
\end{proposition}
\begin{proof}
	We may freely assume that $\delta \geq 1/\log^A N$, since otherwise there is nothing to prove. In particular, $\delta = \log^{O(A)}N$ and $1/\delta = O(\log^A N)$. 
	Decomposing $S$ into a bounded number of pieces, we may assume that $S$ is a basic semialgebraic set. We will assume that $\inter S \neq \emptyset$; the case where $\inter S = \emptyset$ can be handled using similar methods and is somewhat simpler. Thus, $S$ takes the form
	\begin{align}\label{eq:Mobius-def-S}
		S = \set{ (t,x) \in (\RR/\ZZ) \times (G/\Gamma)}{ P_1(t,x) > 0,\ P_2(t,x) > 0,\ \dots, P_r(t,x) > 0 }, 
	\end{align}
	where $r = O(1)$ and $P_i$ are polynomial maps (under identification of $(\RR/\ZZ) \times (G/\Gamma)$ with $[0,1)^{1+D}$) with $\deg P_i = O(1)$ for $1 \leq i \leq r$. Scaling, we may assume that $\norm{P_i}_\infty = 1$ for all $1 \leq i \leq r$. Let $\tau_1$ denote Mal'cev coordinates on $(\RR/\ZZ) \times (G/\Gamma)$, given by $\tau_1(t ,x) = (t,\tau(x))$, where we identify $[0,1)$ with $\RR/\ZZ$ in the standard way.  Furthermore, splitting $S$ further and applying a translation if necessary, we may assume that $\tau_1(S) \subset \bra{\frac{1}{10},\frac{9}{10}}^{1+D}$, implying in particular that $\tau_1$ is continuous in a neighbourhood of $S$.
		
	 Let $\eta \in (0,\delta)$ be a small positive quantity, to be specified in the course of the argument, and let $\Psi,\Psi' \colon \RR \to [0,1]$ be given by
	\[
	\Psi(t) = 
	\begin{cases} 
	0 & \text{if } t < 0,\\
	t/\eta & \text{if } t \in [0,\eta],\\
	1 & \text{if } t > \eta,
	\end{cases}
	\qquad	
	\Psi'(t) = 
	\begin{cases} 
	0 & \text{if } \abs{t} > 2\eta,\\
	2- \abs{t}/\eta & \text{if } \abs{t} \in [\eta,2\eta],\\
	1 & \text{if } \abs{t} < \eta.
	\end{cases}	
	\]
It is clear that $\normLip{\Psi} = \normLip{\Psi'} = 1/\eta$. Let $\Psi_{\square} \colon [0,1)^{1+D} \to [0,1]$ be a $O(1)$-Lipschitz function with $\Psi_{\square}(t,u) = 1$ if $(t,u) \in \bra{\frac{1}{10},\frac{9}{10}}^{1+D}$ and $\Psi_{\square}(t,u) = 0$ if $(t,u) \not\in \bra{\frac{1}{20},\frac{19}{20}}^{1+D}$. For $1 \leq i \leq r$, put
	\begin{align*}
		F_i(t,x) &= \Psi(P_i(t,x)) &&& F_i'(t,x) &= \Psi_{\square}(\tau_1(t,x))\Psi'(P_i(t,x)),\\
		F(t,x) &= \prod_{i=1}^r F_i(t,x) &&& F'(t,x) &= \min\bra{\sum_{i=1}^r F_i'(t,x),1}.
	\end{align*}
	It is routine (although tedious) to verify that $F$ and $F'$ are $1/\eta^{O(1)}$-Lipschitz (cf.\ \cite[Lem.\ A.4]{GreenTao-2012}). Directly from the definitions, we see that for each $t \in \RR/\ZZ$ and $x \in G/\Gamma$ we have $F(t,x) = 1_S(t,x)$ or $F'(t,x) = 1$
. It follows that
\begin{align}\label{eq:Mobius:80:1}
	\abs{ \EEE_{h < H} 
		 \mu(N+h) 1_P(h)1_S\bra{\frac{h}{H},g(h)\Gamma}
	} \leq &
	\abs{ \EEE_{h < H} 
		 \mu(N+h) 1_P(h)F\bra{\frac{h}{H},g(h)\Gamma}
	} \\ \label{eq:Mobius:80:1a} + &
	{ \EEE_{h < H} 
		 F'\bra{\frac{h}{H},g(h)\Gamma}
	}.
\end{align}
	
	In order to estimate either of the summands in \eqref{eq:Mobius:80:1}--\eqref{eq:Mobius:80:1a}, we begin by dividing the interval $[H]$ into $O(1/\a)$ sub-intervals with lengths between $\a H$ and $2\a H$, where 
\begin{equation}\label{eq:Mobius:def-of-alpha}
\a := \bra{\log^A N \max\bra{\normLip{F},\normLip{F'},1}}^{-1} = \eta^{O(1)}/\log^A N
\end{equation}
		To estimate the first summand, we note that for each such sub-interval $[\hh,\hh+H')$, for each $h \in [\hh,\hh+H')$ we have
\begin{align}\label{eq:Mobius:80:15}
	F\bra{\frac{h}{H},g(h)\Gamma} 
	&=  F\bra{\frac{\hh}{H},g(h)\Gamma} + O\bra{\frac{H'}{H} \normLip{F}}
	\\ \label{eq:Mobius:80:15a} 
	&= F\bra{\frac{\hh}{H},g(h)\Gamma} + O\bra{\frac{1}{\log^A N}}.
\end{align}
Applying Theorem \ref{thm:Mobius:MSTT} to each sub-interval, for each constant $C \geq 1$ we obtain
\begin{align}\label{eq:Mobius:80:2}
	\abs{ \EEE_{h < H} 
		 \mu(N+h) 1_P(h)F(g(h)\Gamma)}
	\ll_{C}  \frac{1}{\log^A N} + \frac{1/\eta^{O(1)}}{\log^{C-A}N}.
\end{align}
Let us now consider the second summand. We have, similarly to \eqref{eq:Mobius:80:15},
	\[
	F'\bra{\frac{h}{H},g(h)\Gamma} = F'\bra{\frac{\hh}{H},g(h)\Gamma} + O\bra{\frac{1}{\log^A N}}.
	\]
	 For now, let us assume that $\a > \tilde\delta$, which we will verify at the end of the argument. We conclude from the fact that $(g(h)\Gamma)_{h =0}^{H-1}$ is totally $\tilde\delta$-equidistributed that
\begin{align}\label{eq:Mobius:80:3}
	\EEE_{h \in [\hh,\hh+H')} F'\bra{\frac{h}{H},g(h)\Gamma}
	=& \int_{G/\Gamma} F'\bra{\frac{\hh}{H}, x} dx  + \frac{\tilde\delta}{\eta^{O(1)}} + O\bra{\frac{1}{\log^A N}},
\end{align}	 
where we use $dx$ as a shorthand for $d \mu_{G/\Gamma}(x)$. 
Taking the weighted average of \eqref{eq:Mobius:80:3} over all sub-intervals, we conclude that
\begin{align}\label{eq:Mobius:80:35}
	\EEE_{h < H} F'\bra{\frac{h}{H},g(h)\Gamma}
	= \int_{[0,1)} \int_{G/\Gamma} F'\bra{t, x} dx dt + \frac{\tilde\delta}{\eta^{O(1)}} + O\bra{\frac{1}{\log^A N}}.
\end{align}	 
Applying Lemma \ref{lem:Mobius:poly}\ref{it:Mobius:58:B} to estimate the measure of the support of $F_i'$ for each $1 \leq i \leq r$ we conclude that 
\begin{align}\label{eq:Mobius:80:4a}
	\int_{[0,1)} \int_{G/\Gamma} F'(t,x)  dx dt \ll \eta^{1/O(1)}.
\end{align}
Thus, we may choose $\eta = 1/\log^{O(A)} N$ such that
\begin{align}\label{eq:Mobius:80:4b}
	\int_{[0,1)} \int_{G/\Gamma}F'(t,x) dx dt \leq \frac{1}{\log^A N},
\end{align}
which allows us to simplify \eqref{eq:Mobius:80:35} to 
\begin{align}\label{eq:Mobius:80:36}
	\EEE_{h < H} F'\bra{\frac{h}{H},g(h)\Gamma} =  O\bra{\frac{1}{\log^{B-O(A)}N}} + O\bra{\frac{1}{\log^A N}} .
\end{align}	 
Combining \eqref{eq:Mobius:80:2} and \eqref{eq:Mobius:80:36} with \eqref{eq:Mobius:80:1}--\eqref{eq:Mobius:80:1a}, we conclude that
\begin{align}\label{eq:Mobius:80:5}
	\abs{ \EEE_{h < H} \mu(N+h) 1_P(h)1_S(g(h)\Gamma) } 
	\ll_{C} \frac{1}{\log^{C-O(A)}N} + \frac{1}{\log^{B-O(A)} N} + \frac{1}{\log^A N}.
\end{align}
Letting $C$ and $B$ be sufficiently large multiples of $A$, we conclude that
\begin{align}\label{eq:Mobius:80:6}
	\abs{ \EEE_{h < H} \mu(N+h) 1_P(h)1_S(g(h)\Gamma) } 
	\ll_{A} \frac{1}{\log^{A}N},	
\end{align} 
as needed. Note that choosing $B$ as a large multiple of $A$ also guarantees that $\a = 1/\log^{O(A)}N  > \tilde\delta = 1/\log^B N$.
\end{proof} 
 
\subsection{General case} 
 Before we proceed with the proof of Theorem \ref{thm:Mobius:semialg} in full generality, we will need the following technical lemma. 
\begin{lemma}\label{lem:Mobius:poly}
	Let $d,D \in \NN$, and let $\cV$ denote the vector space of all polynomial maps $P \colon [0,1)^D \to \RR$ of degree at most $d$.
\begin{enumerate}
\item\label{it:Mobius:58:A} There is a constant $C > 1$ (dependent on $d,D$) such that for $P \in \cV$ given by
\[
	P(x) = \sum_{\a \in \NN_0^D} a_\a \prod_{i=1}^D x_i^{\a_i}
\]
we have the inequalities
\(
	C^{-1} \norm{P}_\infty \leq \max_{\a} \abs{a_\a} \leq C \norm{P}_\infty.
\)

\item\label{it:Mobius:58:B} For each $P \in \cV$ and for each $\delta \in (0,1)$ we have
	\begin{equation}\label{eq:Mobious:75:1}
		\lambda\bra{\set{x \in [0,1)^D}{ \abs{P(x)} < \delta^d\norm{P}_\infty}} \ll_{d,D} \delta.
	\end{equation}
\end{enumerate}	
\end{lemma} 
\begin{proof}
	Item \ref{it:Mobius:58:A} follows from the fact that each two norms on the finitely-dimensional vector space $\cV$ are equivalent. For item \ref{it:Mobius:58:B} we proceed by induction with respect to $D$. Multiplying $P$ by a scalar, we may assume that $\norm{P}_\infty = 1$. 
	
	Suppose first that $D = 1$. We proceed by induction on $d$. If $d = 1$ then $P$ is an affine function $P(x) = a x + b$, and the claim follows easily. Assume that $d \geq 2$ and that the claim has been proved for $d-1$. By item \ref{it:Mobius:58:A}, at least one of the coefficients of $P$ has absolute value $\gg_{d,D} 1$. In fact, we may assume that this coefficient is not the constant term, since otherwise for all $x \in [0,1)$ we would have $P(x) \in (\frac{99}{100}P(0),\frac{101}{100}P(0))$ and hence the set in \eqref{eq:Mobious:75:1} would be empty for sufficiently small $\delta$. Thus, $\norm{P'}_\infty \gg_{d,D} 1$. By the inductive assumption, 
	\begin{equation}\label{eq:Mobious:75:2}
		\lambda\bra{\set{x \in [0,1)}{ \abs{P'(x)} < \delta^{d-1} }} \ll_{d} \delta.
	\end{equation}
	Thus, it will suffice to show that
	\begin{equation}\label{eq:Mobious:75:3}
		\lambda\bra{\set{x \in [0,1)}{ \abs{P(x)} < \delta^d,\  \abs{P'(x)} > \delta^{d-1}}} \ll_{d} \delta.
	\end{equation}
	For each interval $I \subset [0,1)$ such that $P'(x)$ has constant sign for $x \in I$ we have
	\begin{equation}\label{eq:Mobious:75:4}
		\lambda\bra{\set{x \in I}{ \abs{P(x)} < \delta^d,\  \abs{P'(x)} > \delta^{d-1}}} \ll \delta.
	\end{equation}
	Since $[0,1)$ can be divided into $O(d)$ intervals where $P$ is monotonous, \eqref{eq:Mobious:75:3} follows.
	
	Suppose now that $D \geq 2$ and the claim has been proved for all $D' < D$. Reasoning like above, we infer from item \ref{it:Mobius:58:A} that $P$ has a coefficient with absolute value $\gg_{d,D} 1$ other than the constant. We may expand $P$ in the form
	\begin{align*}
			P(y,t) &= \sum_{i=0}^{d} t^{i} Q_i(y),
			&&& y \in [0,1)^{D-1},\ t \in [0,1),
	\end{align*}
	where $Q_i$ are polynomials in $D-1$ variables of degree $d-i$. Changing the order of variables if necessary, we may assume that there exists $j$ with $1 \leq j \leq d$ such that $Q_j$ has a coefficient $\gg_{d,D} 1$, and hence $\norm{Q_j}_{\infty} \gg_{d,D} 1$. For $k \in \NN$, let us consider the set
\[
	E_k := \set{ (y,t) \in [0,1)^D }{ 
	\abs{P(y,t)} < \delta^d,\
	2^{-k} \leq \abs{Q_j(y)} < 2^{-k+1} 
	}
\]
The set in \eqref{eq:Mobious:75:4} is the disjoint union $\bigcup_{k=1}^\infty E_i$, so our goal is to show that 
	\begin{equation}\label{eq:Mobious:75:5}
	\sum_{k=1}^\infty \lambda(E_k) \ll_{d,D} \delta.
	\end{equation}
Fix a value of $k$. By the inductive assumption, as long as $j \neq d$, we have
	\begin{equation}\label{eq:Mobious:75:6}
	\lambda\bra{
	\set{y \in [0,1)^{D-1} }{ \abs{Q_j(y)} < 2^{-k+1} }
	}  \ll_{d,D} 2^{-k/(d-j)}.
	\end{equation}
	(If $j = d$, the set in \eqref{eq:Mobious:75:6} is empty for all sufficiently large $k$, and the reasoning simplifies.) 
For each $y \in 	[0,1)^{D-1}$ such that $2^{-k} \leq \abs{Q_j(y)} < 2^{-k+1}$, by the inductive assumption (for $D=1$) we have
	\begin{equation}\label{eq:Mobious:75:7}
	\lambda\bra{
	\set{t \in [0,1) }{ \abs{P(y,t)} < \delta^d }
	}  \ll_{d,D} 2^{k/d} \delta.  
	\end{equation}
	Combining \eqref{eq:Mobious:75:6} and \eqref{eq:Mobious:75:7} yields
	\begin{equation}\label{eq:Mobious:75:8}
	\lambda\bra{E_k} \ll_{d,D} 2^{-k j/d(d-j)} \delta \leq 2^{-k/d^2} \delta.
	\end{equation}
	Summing \eqref{eq:Mobious:75:8} gives \eqref{eq:Mobious:75:5} and finishes the argument.
\end{proof}
  
\begin{proof}[Proof of Theorem \ref{thm:Mobius:1_S-a}]
	The argument is very similar to the proof of Theorem 1.1 assuming Proposition 2.1 in \cite{GreenTao-2012-Mobius}. As the first step, we apply the factorisation theorem \cite[Thm.\ 1.19]{GreenTao-2012}, Theorem \ref{thm:GT-factor}, with $M_0 = \log N$ and parameter $C$ to be determined in the course of the argument. We conclude that there exists an integer $M$ with $\log N \leq M \ll \log^{O_C(1)}N$ such that $g$ admits a factorisation of the form
	\begin{align}\label{eq:Mobius:97:1}
		g(h) = \e(h) g'(h) \gamma(h),
	\end{align}
	where $\e$ is $(M,H)$-smooth, $\gamma$ is $M$-rational, and $g'$ takes values in a rational subgroup $G' < G$ which admits a Mal'cev basis $\cX'$ where each element is a $M$-rational combination of elements of $\cX$, and $(g'(h)\Gamma)_{h =0}^{H-1}$ is totally $1/M^{C}$-equidistributed in $G'/(\Gamma \cap G')$ (with respect to the metric induced by $\cX'$).
	
	With the same reasoning as in \cite{GreenTao-2012-Mobius}, we conclude that $(\gamma(h)\Gamma)_h$ is a periodic sequence with some period $q \leq M$, and for each $0 \leq j < q$ and $h \equiv j \bmod q$ we have $\gamma(h) \Gamma = \gamma_j \Gamma$ for some $\gamma_j \in G$ with coordinates $\tau(\gamma_j)$ that are rationals with height $\ll M^{O(1)}$. Splitting the average in \eqref{eq:Mobius:1_S-a} into sub-progressions, it will suffice to show that for each residue $0 \leq j < q$ modulo $q$, and for each arithmetic progression $Q \subset q \ZZ + j$ with diameter at most $N/M$ we have
	\begin{align}\label{eq:Mobius:97:2}
		\abs{ \EEE_{h < H} \mu(N+h) 1_Q(h) 1_S(\e(h) g'(h) \gamma_j \Gamma) } 
		\ll_A \frac{(1/\delta)^{O(1)}}{M^2 \log^A N}.
	\end{align}
	
	The key difference between our current work and the corresponding argument in \cite{GreenTao-2012-Mobius} is that $1_S$ is not continuous and hence in \eqref{eq:Mobius:97:2} we cannot replace $\e(h)$ with a constant and hope that the value of the average will remain approximately unchanged. Instead, we will use an argument of a more algebraic type. We note that, as a consequence of invariance of the metric on $G$ under multiplication on the right, for each $h,h' \in Q$ we have
	\[
		d\bra{ \e(h)g'(h)\gamma_j, \e(h')g'(h)\gamma_j} =
		d\bra{ \e(h), \e(h')} = O(1). 
	\]
Let us fix $\hh \in Q$ and put $\e'(h) = \e(h)\e(\hh)^{-1}$. Then $d(\e'(h),e_G) = O(1)$ and $g(h)\Gamma = \e(h) g'(h) \gamma_j \Gamma = \e'(h) \e(\hh)g'(h)\gamma_j\Gamma$. 

	Let $\Omega \subset G$ be a bounded semialgebraic set such that $\e'(h) \in \Omega$ for all $h \in Q$. For instance, we may take $\Omega$ to be the pre-image of a certain ball with radius $1/\delta^{O(1)}$ under $\tilde\tau$. Let also $\Pi := \tilde\tau^{-1}\bra{[0,1)^D}$ denote the standard fundamental domain for $G/\Gamma$. Consider the set
	\[
		R = \set{ (g_1,g_2) \in \Omega \times \Pi }{ g_1g_2\Gamma \in S}.
	\] 
We may decompose $R$ as
	\begin{align}\label{eq:Mobius:97:5}
		R &= \bigcup_{\gamma \in \Gamma} R_\gamma && \text{ where } & 
		R_\gamma = \set{ (g_1,g_2) \in \Omega \times \Pi }{ g_1g_2 \Gamma \in S,\ g_1g_2\gamma \in \Pi}.
	\end{align}
	Using the quantitative bounds in \cite[Lem.\ A.2 \& A.3]{GreenTao-2012}, we see that for each $\gamma \in \Gamma$ such that $R_\gamma \neq \emptyset$ we have $\abs{\tilde\tau(\gamma)} = O(1/\delta^{O(1)})$. Hence, the union in \eqref{eq:Mobius:97:5} involves $O(1/\delta^{O(1)})$ non-empty terms, and in particular is finite. Each of the sets $R_\gamma$ is semialgebraic with complexity $O(1)$. Moreover, since $\e'$ is a polynomial map of bounded degree, for each $\gamma \in \Gamma$ the set
	\[
		T_\gamma = \set{ (t,x) \in [0,1) \times \Pi }{ \bra{\e'(tH),x} \in R_\gamma}
	\]
	is also semialgebraic with complexity $O(1)$. Hence, \eqref{eq:Mobius:97:2} will follow once we show that for each semialgebraic set $T \subset [0,1) \times G/\Gamma$ with bounded complexity we have
	\begin{align}\label{eq:Mobius:97:6}
		\abs{ \EEE_{h < H} \mu(N+h) 1_Q(h) 1_{T}\bra{ \frac{h}{H}, \e(\hh) g'(h) \gamma_j \Gamma} } 
		\ll_A \frac{(1/\delta)^{O(1)}}{M^2 \log^A N}.
	\end{align}

\newcommand{\ttilde}[1]{\tilde{#1}'}
	Following \cite{GreenTao-2012-Mobius}, we put $\ttilde{G} := \gamma_j^{-1}G'\gamma_j$, $\Lambda := \Gamma \cap \ttilde{G}$ and $ \ttilde{g}(n) := \gamma_j^{-1}g'(n)\gamma_j$. Let also $D' = \dim G'$, let $\sigma$ and $\tilde\sigma$ denote the coordinate maps on $\ttilde{G}/\Lambda$ and $\ttilde{G}$ respectively, and let $\Delta = \tilde \sigma^{-1}\bra{[0,1)^{D'}}$ denote the fundamental domain. Then $\ttilde{g}$ is a polynomial sequence with respect to the filtration $\ttilde{G}_{\bullet}$ given by $\ttilde{G}_{i} = \gamma_j^{-1}G'_i\gamma_j$. We have a well-defined map $\iota \colon \ttilde{G}/\Lambda \to G/\Gamma$ given by 
	\[
		\iota(x\Lambda) = \e(\hh)\gamma_j x \Gamma.
	\]
	Thus, for all $h \in [H]$ we have
	\[
		\e(\hh) g'(h) \gamma_j \Gamma = \iota( \ttilde{g}(h) \Lambda)
	\]
	As discussed in \cite{GreenTao-2012}, the Lipschitz norm of the map $\iota$ is $O(M^{O(1)})$ and the sequence $(\ttilde{G}(h)\Lambda)_{h = 0}^{H-1}$ is $1/M^{\lambda C+O(1)}$-equidistributed, where $\lambda > 0$ is a constant dependent only on $d$ and $D$. 
	
	For each $\gamma \in \Gamma$, the map $\iota$ is a polynomial on the semialgebraic set $\Delta \cap \iota^{-1}(\Pi \gamma)$. The estimate on the Lipschitz norm of $\iota$ implies that $\Delta$ can be partitioned into $M^{O(1)}$ semialgebraic sets with complexity $O(1)$ such that, on each of the pieces $\iota$ is a polynomial of degree $O(1)$ (using the coordinates $\tilde \tau$ and $\tilde \sigma$). Applying the corresponding partition in \eqref{eq:Mobius:97:6}, we see that it will suffice to show that for each semialgebraic set $T \subset (\RR/\ZZ) \times (\ttilde{G}/\Lambda)$ with bounded complexity and for each constant $A' > 0$ we have
	\begin{align}\label{eq:Mobius:97:7}
		\abs{ \EEE_{h < H} \mu(N+h) 1_Q(h) 1_{T}\bra{ \frac{h}{H}, g_j(h) \Lambda} } 
		\ll_{A,A'} \frac{(1/\delta)^{O(1)}}{M^{A'} \log^A N}.
	\end{align}
Bearing in mind that $M \geq \log N$, it will suffice to show that
	\begin{align}\label{eq:Mobius:97:7a}
		\abs{ \EEE_{h < H} \mu(N+h) 1_Q(h) 1_{T}\bra{ \frac{h}{H}, g_j(h) \Lambda} } 
		\ll_{A} \frac{(1/\delta)^{O(1)}}{M^{A}}.
	\end{align}
We are now in position to apply Proposition \ref{thm:Mobius:1_S-b} on $\ttilde{G}/\Lambda$. The complexity of $(\ttilde{G},\Lambda,\ttilde{G}_\bullet)$ is $1/\delta'$, where $\delta' = 1/M^{O(1)}$. The largest exponent $A'$ with which Proposition \ref{thm:Mobius:1_S-b} is applicable to $(\ttilde{g}(h))_{h=0}^{H-1}$ satisfies $\log^{A'}N \gg M^{\mu C}$ for a constant $\mu \gg 1$, leading to 
	\begin{align}\label{eq:Mobius:97:8}
		\abs{ \EEE_{h < H} \mu(N+h) 1_Q(h) 1_{T}\bra{ \frac{h}{H}, g_j(h) \Lambda} } 
		\ll_{C} \frac{1}{M^{\mu C - O(1)}}.
	\end{align}
	In order to derive \eqref{eq:Mobius:97:7a} it is enough to let $C$ be a sufficiently large multiple of $A$.
\end{proof}

\bibliographystyle{alphaabbr}
\bibliography{bibliography}
\end{document}